\documentclass[11pt]{amsart}

\usepackage{amsfonts}
\usepackage{amssymb}
\usepackage{amsmath, amsthm,color}
\usepackage{hyperref}
\usepackage{pdfsync}
\usepackage{bbm, dsfont}
\usepackage[margin=0.96in]{geometry}
\usepackage{enumerate}   
\usepackage[shortlabels]{enumitem}
\usepackage{color}
\usepackage{multicol}
\usepackage{graphicx}

\newcommand{\La}{\langle}
\newcommand{\Ra}{\rangle}

\newcommand{\pd}{\partial}

\newcommand{\eps}{\varepsilon}
\newcommand{\la}{\lambda}

\newcommand{\unit}{1\!\!1}

\newcommand{\Xp}{X^{+}_{\tau, x}}
\newcommand{\Xn}{X^{-}_{\tau, x}}

\newcommand{\bL}{\mathbb{L}}

\newcommand{\bM}{\mathbb{M}}

\newcommand{\bR}{\mathbb{R}}

\newcommand{\cD}{\mathcal{D}}

\newcommand{\cI}{\mathcal{I}}
\newtheorem{theorem}{Theorem}[section]
\newtheorem{corollary}[theorem]{Corollary}
\newtheorem{lemma}[theorem]{Lemma}
\newtheorem{prop}[theorem]{Proposition}
\newtheorem{definition}{Definition}
\newtheorem{remark}[theorem]{Remark}

\numberwithin{equation}{section}

\begin{document}
\title[The Sharp Constant in the Weak (1,1) Inequality for the Square Function]{The Sharp Constant in the Weak (1,1) Inequality for the Square Function: A New Proof}
\author[I.~Holmes, P.~Ivanisvili, A.~Volberg]{I.~Holmes, P.~Ivanisvili, A.~Volberg}
\thanks{I. Holmes is supported by National Science Foundation as an NSF Postdoc under Award No.1606270, A.~Volberg is partially supported by the NSF DMS-1600065.  This paper is  based upon work supported by the National Science Foundation under Grant No. DMS-1440140 while two of the authors, P. Ivanisvili and A. Volberg,  were in residence at the Mathematical Sciences Research Institute in Berkeley, California, during the Spring and Fall  2017 semester. }
%\address{Department of Mathematics, Kent State University, OH 44240, USA}

%\address{Department of Mathematics, Kent State University, OH 44240, USA}
%\email{nazarov@math.kent.edu \textrm{(F.\ Nazarov)}}
%\author{I. Holmes, A. Volberg}
\address{Department of Mathematics, Michigan State University, East Lansing, MI 48823, USA}
\email{holmesir@math.msu.edu \textrm{(I.\ Holmes)}}
\address{Department of Mathematics, Princeton University; MSRI; UC Irvine, CA, USA}
\email{ paata.ivanisvili@princeton.edu \textrm{(P.\ Ivanisvili)}}
\address{Department of Mathematics, Michigan State University, East Lansing, MI 48823, USA}
\email{volberg@math.msu.edu \textrm{(A.\ Volberg)}}

\makeatletter
\@namedef{subjclassname@2010}{
  \textup{2010} Mathematics Subject Classification}
\makeatother

\begin{abstract} 
In this note we give a new proof of the sharp constant $C = e^{-1/2} + \int_0^1 e^{-x^2/2}\,dx$ in the weak (1, 1) inequality for the dyadic square function. The proof makes use of two Bellman functions $\bL$ and $\bM$ related to the problem, and relies on certain relationships between $\bL$ and $\bM$, as well as the boundary values of these functions, which we find explicitly. Moreover, these Bellman functions exhibit an interesting behavior: the boundary solution for $\bM$ yields the optimal obstacle condition for $\bL$, and vice versa.
\end{abstract}
\maketitle

%-------------------------------------%
%-------------Section: ---------------%
%-------------------------------------%

\section{Introduction}

In this paper we consider weak inequalities for the dyadic square function:
	$$
	S\varphi(x) := \left( \sum_{I \in \cD} (\varphi, h_I)^2 \frac{\unit_I(x)}{|I|} \right)^{1/2},
	$$
where $(\cdot, \cdot)$ denotes the usual inner product in $L^2(\bR)$, 
$\cD$ is the standard collection of dyadic intervals on the real line, 
and $\{h_I\}_{I\in\cD}$ are the ($L^2$--normalized) Haar functions:
	$$ h_I(x) := \frac{1}{\sqrt{|I|}} (\unit_{I_{-}}(x) - \unit_{I_{+}}(x)), $$
where $I_{-}$ and $I_{+}$ denote the left and right halves of $I$, respectively.
In particular, we look at localized versions of $S$, applied to compactly supported functions;
for a dyadic interval $J \in \cD$, let
	$$
	S_J^2 \varphi := \sum_{I \subset D, \; I \subseteq  J} (\varphi, h_I)^2 \frac{\unit_I}{|I|}
		= \sum_{I \subset J} |\Delta_I\varphi|^2 \unit_I,
	$$
where $\Delta_I\varphi$ denotes the martingale difference 
	$$\Delta_I\varphi := \frac{1}{2}(\La\varphi\Ra_{I_{+}} - \La\varphi\Ra_{I_{-}})(\unit_{I_{+}} - \unit_{I_{-}})
		= (\varphi, h_I) h_I. $$
Note that 
	$ S_J^2 \varphi = S [(\varphi - \La\varphi\Ra_J)\unit_J ], $
where $\La \varphi \Ra_J := \frac{1}{|J|} \int_J \varphi \,dx$,
so we may always assume that supp$(\varphi) \subset J$.

We are looking for the sharp constant $C$ in the inequality
	$$|\{x\in J: S_J^2\varphi(x) \geq \lambda\}| \leq C \frac{1}{\sqrt{\lambda}}\int_J|\varphi|,$$
for all $\varphi \in L^1(J)$ and $J \in \cD$. It was conjectured by Bollobas in \cite{Bollobas},
and it was later proved by Osekowski in \cite{Osekowski}, that this constant is
	\begin{equation}\label{E:SharpC}
	C = \Psi(1) \text{, where } \Psi(\tau) = \tau \Phi(\tau) + e^{-\tau^2/2} \text{ and } \Phi(\tau) = \int_0^{\tau} e^{-x^2/2}\,dx.
	\end{equation}
In this paper we give a new proof of this fact, using Bellman functions. We use several of them, and, roughly speaking, try to solve an obstacle problems
for the PDE that are assigned to these Bellman functions. 

As often happens in obstacle problems, the solution breaks the domain of definition to two sub-domains: 
the first one is where the solution is equal to the obstacle, and the second one, where the solution is strictly bigger 
(or strictly smaller, depending on the problem) than the obstacle, and in this domain
the corresponding PDE should be solved precisely. This can be a difficult task (we deal with fully nonlinear degenerate  elliptic equations), but the sharp constant in the underlying inequality 
can be found sometimes without fulfilling this difficult task in its entirety. This is what we will be doing. But we found out the 
sub-domains mentioned above, so we found precisely, where our Bellman functions coincide with a corresponding obstacle functions.

\medskip

B. Bollobas published \cite{Bollobas} in 1982 but apparently he initiated this problem in mid-70's, as it is said in \cite{Bollobas} that he invented
the problem to entertain professor Littlewood. In \cite{Bollobas} a certain constant and a certain special function (the Bellman function of an underlying problem)
were invented. But the fact that the constant and the function of Bollobas are precisely the best constant and the Bellman function correspondingly were proved only in 2008 by A. Osekowski in \cite{Osekowski}. We give here a different proof of this fact, and we list also some extra properties of the function found by Bollobas in \cite{Bollobas}.

\medskip

In Section \ref{S:PropsBellman}  we begin by defining the standard Bellman function for the above listed problem:
\begin{definition} \label{D:bM}
Given $f \in \bR$, $F \geq |f|$, and $\lambda > 0$, define:
	$$\bM(f, F, \lambda) := \sup \frac{1}{|J|}|\{x\in J: S_J^2\varphi(x) \geq \lambda\}|,$$
where the supremum is over all functions $\varphi$, supported in $J \in \cD$, 
such that $\La\varphi\Ra_J = f$ and $\La |\varphi|\Ra_J = F$.
We say that any such $\varphi$ is an admissible function for $\bM(f, F, \lambda)$.
\end{definition}

As shown in Proposition \ref{P:bM-Props}, this function has the expected properties, such as a main inequality and an obstacle condition.
Also as expected, we show in Theorem \ref{T:bM-LstSup} that $\bM$ is the so-called ``least supersolution'' for its main inequality.

Next, we define another Bellman function, also associated to this problem:

\begin{definition} \label{D:bL}
Given $f \in \bR$, $0 \leq p \leq 1$, and $\lambda > 0$, define:
	$$\bL (f, p, \lambda) := \inf \La |\varphi|\Ra_J,$$
where the infimum is over all functions $\varphi$, supported in $J \in \cD$,
such that
	$$\La\varphi\Ra_J = f \:\:\text{ and }\:\: \frac{1}{|J|}|\{x \in J: S_J^2\varphi(x) \geq \lambda\}| = p.$$
We say that any such $\varphi$ is an admissible function for $\bL(f, p, \lambda)$.
\end{definition}

This definition is inspired by Bollobas \cite{Bollobas} -- see Remark \ref{R:BollDef} for details of the connection to Bollobas's definition.
Being defined as an infimum, this function will have most of the mirrored properties of $\bM$ -- replace concavity with convexity for example. These are detailed in Proposition \ref{P:bL-Props}. Also mirroring $\bM$, we show in Theorem \ref{T:bL-GrtSub} that $\bL$ is the so-called ``greatest subsolution'' for its main inequality.

Using the standard methods, we obtain so-called ``obstacle conditions'' for $\bM$ and $\bL$, namely
	$$\bM(f, F, \lambda) = 1, \:\forall F \geq \sqrt{\lambda} \:\:\text{ and }\:\: 
		\bL(f, p, \lambda) = |f|, \:\forall |f| \geq \sqrt{\lambda}.$$
		
		\medskip
		
While these obstacle conditions suffice, as expected, to prove the least supersolution and greatest subsolution results, there is no reason to believe these obstacle conditions are optimal. That is, $\bM$ could very well be equal to $1$ for some points where $F < \sqrt{\lambda}$, for instance. As it turns out, we may find out  the optimal (largest) domains where obstacle condition for $\bM$ holds from information about $\bL$, and vice versa. 

\medskip

In Section \ref{S:RelsML} we explore the connections between $\bM$ and $\bL$. We show in Theorem \ref{T:MLRel} that 
$\bL(f, p, \lambda)$ is the smallest value of $F$ for which $\bM(f, F, \lambda) \geq p$, 
and $\bM(f, F, \lambda)$ is the largest value of $p$ such that $\bL(f, p, \lambda) \leq F$:
	$$ \bL(f, p, \lambda) = \inf\{F \geq |f|: \bM(f, F, \lambda) \geq p\} \:\text{ and }\: 
	\bM(f, F, \lambda) = \sup\{p\in [0, 1]: \bL(f, p, \lambda) \leq F\}.
	$$
These relationships are further improved in Proposition \ref{P:ML-Equals}, where we show that in certain domains
(ultimately the really ``interesting'' parts of the domains), we have in fact that
	$$\bM\big(f, \bL(f, p, \lambda), \lambda \big) = p \:\:\text{ and }\:\:
		\bL\big( f, \bM(f, F, \lambda), \lambda) = F.
	$$
Then the value of $\bM$ along the boundary $F = |f|$:
	$$\bM_b(f, \lambda) := \bM(f, |f|, \lambda) = \sup\{p \in [0,1]:\: \bL(f, p, \lambda) = |f|\},$$
yields the optimal obstacle condition for $\bL$, and the value of $\bL$ along the boundary $p=1$:
	$$\bL_b(f, \lambda) := \bL(f, 1, \lambda) = \inf\{F \geq |f|:\: \bM(f, F, \lambda) = 1\},$$
yields the optimal obstacle condition for $\bM$. We find $\bM_b$ and $\bL_b$ explicitly in Section \ref{S:Boundaries}.
See Section \ref{Ss:OptimalOC} and Figures \ref{fig1} and \ref{fig2} for a description of the optimal obstacle conditions for $\bM$ and $\bL$ obtained from these boundary values.

In Section \ref{S:Sharp} we give the new proof of the sharp constant in \eqref{E:SharpC}. The inequality 
	$$\bM(0, F, \lambda) \leq \frac{F}{\bL(0, 1, \lambda)} = \frac{F}{\bL_b(0, \lambda)},$$
got detailed proof in Theorem \ref{T:Ineq-Bol}.
This, combined with the relationship
	$$\bL(f, \bM(f, F, \lambda), \lambda) = F,$$
and the expression of $\bL_b$ obtained in Theorem \ref{T:Lb-Bdry}, then yields the desired sharp constant $C$,
as detailed in Corollary \ref{C:SharpC}. The proof is significantly simplified once we find, in Proposition \ref{P:Bdryf0},
the values of $\bM$ and $\bL$ at $f = 0$.

\medskip

\noindent{\bf Acknowledgements.} The authors would like to thank the anonymous referee for a very careful reading of our paper, and for the suggestions which greatly improved the work.

%--------------------------%
%---------Section----------%
%--------------------------%

\section{Properties of the Bellman Functions $\bM$ and $\bL$} 
\label{S:PropsBellman}

%-----------------------------%
%---------SubSection----------%
\subsection{Basic Properties.}\label{Ss:BasicProps} 
In this section we prove the basic properties of $\bM$ and $\bL$, such as the main inequalities, convexity, monotonicity, and obstacle conditions.

%--------Basic Properties of bM--------%
\begin{prop} \label{P:bM-Props}
The Bellman function $\bM(f, F, \lambda)$ in Definition \ref{D:bM} has the following properties:
	\begin{enumerate}[\bf1)]
	\item $\bM$ is independent of the choice of interval $J \in \cD$ in its definition.
	\item \textbf{Domain and Range:} $\bM$ has convex domain
		$\Omega_{\bM} := \{(f, F, \lambda) : |f| \leq F; \: \lambda > 0\}$, 
		and $0 \leq \bM \leq 1$.
	\item $\bM$ is decreasing in $\lambda$.
	\item $\bM$ is even in $f$.
	\item {\bf Homogeneity:} 
		\begin{equation} \label{E:Hom-bM}
		\bM(f, F, \lambda) = \bM(tf, |t|F, t^2\lambda), \: \forall t \neq 0.
		\end{equation}
	\item {\bf Obstacle Condition:}
		\begin{equation} \label{E:OC-bM}
		\bM(f, F, \lambda) = 1, \: \forall \lambda \leq F^2.
		\end{equation}
	\item {\bf Main Inequality:} For all triplets $(f, F, \lambda)$, $(f_{\pm}, F_{\pm}, \lambda_{\pm})$ in the domain with
		$f = \frac{1}{2}(f_{-}+ f_{+})$, $F = \frac{1}{2}(F_{-}+ F_{+})$, and $\lambda = \min(\lambda_{-}, \lambda_{+})$,
		there holds:
			\begin{equation} \label{E:MI-bM}
			\bM\bigg(f, F, \lambda + \bigg(\frac{f_{+}-f_{-}}{2}\bigg)^2 \bigg) \geq \frac{1}{2} 
			\bigg(\bM(f_{+}, F_{+}, \lambda_{+}) + \bM(f_{-}, F_{-}, \lambda_{-}) \bigg).
			\end{equation}
	\item $\bM$ is concave in the variables $f$ and $F$.
	\item $\bM$ is maximal at $f = 0$:
		\begin{equation}\label{E:bM-max0}
		\bM(f, F, \lambda) \leq \bM(f, F, \lambda - f^2) \leq \bM(0, F, \lambda).
		\end{equation}
	\item $\bM$ is non-decreasing in $F$; $\bM$ is non-increasing in $f$ for $f \geq 0$ (and non-decreasing in $f$ for $f \leq 0$).
	\end{enumerate}
\end{prop}

\begin{proof}
\textbf{1)} follows by the standard considerations. Properties {\bf 2)} and {\bf 3)} are obvious.
Property {\bf 4)} follows since $\varphi$ is admissible for $\bM(f, F, \lambda)$ if and only if $-\varphi$ is admissible for $\bM(-f, F, \lambda)$,
and in this case $S_J^2\varphi = S_J^2(-\varphi)$. 
To see homogeneity, {\bf 5)}, note that $\varphi$ is admissible for $\bM(f, F, \lambda)$ if and only if $t \varphi$ is admissible for
$\bM(tf, |t|F, t^2\lambda)$, and in this case $S_J^2(t \varphi) = t^2 S_J^2\varphi$.

Next, we prove the obstacle condition, {\bf 6)} Given a point $(f, F, \lambda)$ in the domain, consider the function
$\varphi = f \unit_J + F\sqrt{|J|}h_J$. Then $\La \varphi \Ra_J = f$, $S_J^2\varphi = F^2\unit_J$, and
$\La |\varphi|\Ra_J = F$. So $\varphi$ is admissible for $\bM(f, F, \lambda)$, and if $\lambda \leq F^2$,
$\{x\in J: S_J^2\varphi(x) \geq \lambda\} = J$, so $\bM(f, F, \lambda) = 1$.

To prove the main inequality {\bf 7)}, let $J \in \cD$ be a dyadic interval, and let $\varphi_{\pm}$ be functions supported on $J_{\pm}$, admissible for
$\bM(f_{\pm}, F_{\pm}, \lambda_{\pm})$, and which give the supremum up to some $\epsilon > 0$:
	$$\text{supp}(\varphi_{\pm}) \subset J_{\pm}; \:\: \La \varphi_{\pm}\Ra_{J_{\pm}} = f_{\pm}; \:\:
		\La |\varphi_{\pm}|\Ra_{J_{\pm}} = F_{\pm}, $$
and
	$$\frac{1}{|J_{\pm}|} \left|\{ x \in J_{\pm}: S_{J_{\pm}}^2\varphi_{\pm} \geq \lambda_{\pm} \}\right| 
		> \bM(f_{\pm}, F_{\pm}, \lambda_{\pm}) - \epsilon.$$
Now define $\varphi$ on $J$ by concatenation: $\varphi := \varphi_{-}\unit_{J_{-}} + \varphi_{+}\unit_{J_{+}}$.
Then $\La \varphi\Ra_J = f$ and $\La |\varphi|\Ra_J = F$, so $\varphi$ is admissible for $\bM(f, F, \lambda)$.
Moreover:
	\begin{align*}
	S_J^2\varphi &= |\Delta_J\varphi|^2\unit_J + \sum_{I \subset J_{-}} |\Delta_I \varphi_{-}|^2 \unit_I
		+ \sum_{I \subset J_{+}} |\Delta_I\varphi_{+}|^2 \unit_I \\
	&= \frac{1}{4}(f_{+} - f_{-})^2 \unit_J + S_{J_{-}}^2\varphi_{-} + S_{J_{+}}^2\varphi_{+}.
	\end{align*}
Then:
	\begin{align*}
	\bM(f, F, \lambda + \frac{1}{4}(f_{+} - f_{-})^2) &\geq 
		\frac{1}{|J|} \left|\{ x \in J: S_{J_{-}}^2\varphi_{-}(x) + S_{J_{+}}^2\varphi_{+}(x) \geq \lambda \}\right| \\
	&\geq \frac{1}{2|J_{-}|} \left|\{ x \in J_{-} : S_{J_{-}}^2\varphi_{-}(x) \geq \lambda_{-} \}\right| +
		\frac{1}{2|J_{+}|} \left|\{ x \in J_{+} : S_{J_{+}}^2\varphi_{+}(x) \geq \lambda_{+} \}\right| \\
	&> \frac{1}{2} \left( \bM(f_{-}, F_{-}, \lambda_{-}) + \bM(f_{+}, F_{+}, \lambda_{+}) \right) - \epsilon.
	\end{align*}
Since this holds for all $\epsilon > 0$, the main inequality \eqref{E:MI-bM} is proved.

To prove {\bf 8)}, rewrite the main inequality in a more convenient form:
	\begin{equation}\label{E:MI-bM2}
	\frac{1}{2}\bigg(\bM(f+a, F+b, \lambda) + \bM(f-a, F-b, \lambda)\bigg) \leq \bM(f, F, \lambda + a^2) \leq \bM(f, F, \lambda),
	\end{equation}
for all $a \in \bR$ and $|b| \leq F$. We immediately obtain that $\bM$ is \textit{midpoint concave}
in the variable $f, F$. Since $\bM$ is measurable, this is enough to show that $\bM$ is concave
in $F$ and $f$ (see page 60 in \cite{ConvexBook}, and the references therein \cite{Blumberg, Sierpinski}).

For {\bf 9)}, take $f = 0$ and $b = 0$ in \eqref{E:MI-bM2}:
	$$\bM(0, F, \lambda + a^2) \geq \frac{1}{2}\bigg(\bM(a, F, \lambda) + \bM(-a, F, \lambda)\bigg) = \bM(a, F, \lambda),$$
where the last equality follows because $\bM$ is even in the first variable.

Finally, to see {\bf 10)} note that by the obstacle condition \eqref{E:OC-bM}, $\bM(f, \cdot, \lambda)$ is concave and has a maximum at $F = \sqrt{\lambda}$, and is constant for $F \geq \sqrt{\lambda}$. Similarly, $\bM(\cdot, F, \lambda)$ is even, concave, and by \eqref{E:bM-max0} has a maximum at $f = 0$.
\end{proof}

\noindent Note that if $F=0$, the only admissible function is $\varphi = 0$ a.e. so
	\begin{equation}\label{E:bM-F0}
	\bM(0, 0, \lambda) = 0, \: \forall \lambda >0.
	\end{equation}

%--------Basic Properties of bL--------%

\begin{prop} \label{P:bL-Props}
The Bellman function $\bL(f, p, \lambda)$ in Definition \ref{D:bL} has the following properties:
	\begin{enumerate}[\bf1)]
	\item $\bL$ is independent of the choice of interval $J \in \cD$ in its definition.
	\item {\bf Domain and Range:} $\bL$ has convex domain $\Omega_{\bL} := \{(f, p, \lambda): f \in \bR; p \in [0,1], \lambda > 0\}$.
	As for the range:
		\begin{equation}\label{E:Range-bL}
		|f| \leq \bL(f, p, \lambda) \leq (1-p)|f| + p\max(|f|, \sqrt{\lambda}).
		\end{equation}
	\item $\bL$ is increasing in $\lambda$.
	\item $\bL$ is even in $f$.
	\item {\bf Homogeneity:} 
		\begin{equation}\label{E:Hom-bL}
		\bL(tf, p, t^2\lambda) = |t| \bL(f, p, \lambda), \: \forall t \neq 0.
		\end{equation}
	\item {\bf Obstacle Condition:}
		\begin{equation} \label{E:OC-bL}
		\bL(f, p, \lambda) = |f|, \: \forall |f| \geq \sqrt{\lambda}.
		\end{equation}
	\item {\bf Main Inequality:} For all triplets $(f, p, \lambda)$, $(f_{\pm}, p_{\pm}, \lambda_{\pm})$ in the domain with
		$f = \frac{1}{2}(f_{-}+ f_{+})$, $p = \frac{1}{2}(p_{-}+ p_{+})$, and $\lambda = \min(\lambda_{-}, \lambda_{+})$,
		there holds:
			\begin{equation} \label{E:MI-bL}
			\bL\bigg(f, p, \lambda + \bigg(\frac{f_{+}-f_{-}}{2}\bigg)^2 \bigg) \leq \frac{1}{2} 
			\bigg(\bL(f_{+}, p_{+}, \lambda_{+}) + \bL(f_{-}, p_{-}, \lambda_{-}) \bigg).
			\end{equation}
	\item $\bL$ is convex in the variables $f, p$.
	\item $\bL$ is minimal at $f = 0$:
		\begin{equation} \label{E:bL-min0}
		\bL(0, p, \lambda) \leq \bL(0, p, \lambda + f^2) \leq \bL(f, p, \lambda).
		\end{equation}
	\item $\bL$ is non-decreasing in $p$; $\bL$ is non-decreasing in $f$ for $f \geq 0$ (and non-increasing in $f$ for $f \leq 0$).
	\end{enumerate}
\end{prop}

\begin{proof}
The proofs of properties {\bf 1)}, {\bf 4)} and {\bf 5)} are similar to those for $\bM$. 
It is also straightforward to prove
	\begin{equation}\label{E:MI-bL2}
	\bL(f, p, \lambda ) \leq \bL(f, p, \lambda + a^2) \leq \frac{1}{2}\bigg(\bL(f+a, p+b, \lambda) + \bL(f-a, p-b, \lambda)\bigg),
	\end{equation}
a weaker form of \eqref{E:MI-bL} -- note that we don't know yet that $\bL$ is increasing in $\lambda$, a property that is not so obvious in this case. We may see now however that $\bL$ is convex  in $p$, by letting $a=0$ in \eqref{E:MI-bL2}.

Next, we prove the range condition {\bf 2)} \eqref{E:Range-bL}, and note that in this case the obstacle condition {\bf 6)} \eqref{E:OC-bL}
follows directly from the range condition, since
	$$|f| \leq \bL(f, p, \lambda) \leq (1-p)|f| + p\max(|f|, \sqrt{\lambda}) \leq \max(|f|, \sqrt{\lambda}).$$
The first inequality is obvious, as any function $\varphi$ admissible for
$\bL(f, p, \lambda)$ satisfies $\La|\varphi|\Ra_J \geq |\La\varphi\Ra_J| = |f|$.
We now prove the second inequality, and begin with some simple examples. When $p = 1$, consider the function
$\varphi = f\unit_J + \sqrt{\lambda}\sqrt{|J|}h_J$. Then $S_J^2\varphi = \lambda \unit_J$, so $\varphi$ is admissible for 
$\bL(f, 1, \lambda)$, and then:
	$$\bL(f, 1, \lambda) \leq \La|\varphi|\Ra_J = \frac{1}{2}|f + \sqrt{\lambda}| + \frac{1}{2}|f - \sqrt{\lambda}|
	= \max\{|f|, \sqrt{\lambda}\}.$$
If $p = \frac{1}{2}$, then consider for example the function $\varphi = f\unit_J + \sqrt{\lambda}\sqrt{|J_{-}|}h_{J_{-}}$. Then
$S_J^2\varphi = \lambda\unit_{J_{-}}$, so $\varphi$ is admissible for $\bL(f, 1/2, \lambda)$, and then:
	$$\bL(f, 1/2, \lambda) \leq \La |\varphi|\Ra_J = \frac{1}{4}|f + \sqrt{\lambda}| + \frac{1}{4}|f - \sqrt{\lambda}|
	+ \frac{1}{2}|f| = \frac{1}{2}|f| + \frac{1}{2} \max\{|f|, \sqrt{\lambda}\}.$$
	
Now suppose that $p \in (0, 1)$ is a \textit{dyadic rational}, that is $p = \frac{k}{2^N}$ for some integers $N \geq 1$
and $1\leq k \leq 2^N - 1$. On some dyadic interval $J$, let $\mathcal{I}$ denote any collection of $k$ subintervals in the
$N^{\text{th}}$ generation $J_{(N)}$ of dyadic descendants of $J$, and let
	$$\varphi = f\unit_J + \sqrt{\lambda} \sum_{I \in \cI} \sqrt{|I|}h_I.$$
Then $S_J^2\varphi = \lambda\unit_{\{\cup I: I \in\cI\}}$, so
	$$\frac{1}{|J|}|\{x \in J: S_J^2\varphi(x)\geq\lambda\}| = \frac{|\bigcup_{I \in \cI}I|}{|J|} = \frac{k}{2^N} = p,$$
and $\bL(f, p, \lambda) \leq \La|\varphi|\Ra_J$. Now for every $I \in \cI$, on $I_{\pm}$, $\varphi = f \pm \sqrt{\lambda}$, 
and $\varphi = f$ off $\cup_{I\in\cI}I$. So
	$$\La|\varphi|\Ra_J = \frac{1}{|J|}\bigg(\max\{|f|, \sqrt{\lambda}\} \sum_{I\in\cI}|I| + 
	|f| |J \setminus \cup_{I\in\cI}I| \bigg) = (1-p)|f| + p \max\{|f|, \sqrt{\lambda}\}.$$ 
Therefore the second inequality in \eqref{E:Range-bL} holds for all dyadic rationals $p \in (0, 1)$, the result follows. 

Also note that, taking $p = 0$ in \eqref{E:Range-bL}, we see that
	\begin{equation}\label{E:bL-p0}
	\bL(f, 0, \lambda) = |f|.
	\end{equation}
Thus $\bL(f, \cdot, \lambda)$ is convex in $p \in [0,1]$ and has a minimum at $p=0$, so $\bL$ is non-decreasing in $p$.
In turn, this allows us to prove property {\bf 3)}, that $\bL$ is non-decreasing in $\lambda$: suppose $\lambda_1 \leq \lambda_2$
and let $\varphi$ be admissible for $\bL(f, p, \lambda_2)$. Then $\La \varphi\Ra_J = f$ and
	$$p = \frac{1}{|J|}|\{x\in J: S_J^2\varphi(x) \geq \lambda_2\}| \leq \frac{1}{|J|}|\{x\in J: S_J^2\varphi(x) \geq \lambda_1\}| =: q.$$
So $\varphi$ is also admissible for $\bL(f, q,\lambda_1)$, where $q \geq p$, which means
	$$\La|\varphi|\Ra_J \geq \bL(f, q, \lambda_1) \geq \bL(f, p, \lambda_1).$$
Since this holds for all $\varphi$ admissible for $\bL(f, p, \lambda_2)$, we have $\bL(f, p, \lambda_2) \geq \bL(f, p, \lambda_1)$.

Having the desired monotonicity in $\lambda$ then gives the full form of the main inequality {\bf 7)} \eqref{E:MI-bL}, as well as
$\bL(f, p, \lambda) \leq \bL(f, p, \lambda+a^2)$. So  \eqref{E:MI-bL2}  gives us convexity in $f, p$ -- so property {\bf 8)} is also proved.
Let $f = 0$ and $b = 0$ in \eqref{E:MI-bL2} and we obtain {\bf 9)}, minimality of $\bL$ at $f = 0$.
Finally, we may then finish proving {\bf 10)}: since $\bL(\cdot, p,\lambda)$ is even, convex, and minimal at $f = 0$,
the claimed monotonicity in $f$ follows.
	
\end{proof}

%-----------------------------%
%---------SubSection----------%
\subsection{$\bM$ is the Least Supersolution.} \label{Ss:bM-LstSup}
Consider the main inequality for $\bM$ in more generality:
	\begin{equation}\label{E:MI-m}
	m(f, F, \lambda + a^2) \geq \frac{1}{2} \bigg(m(f+a, F+b, \lambda) + m(f-a, F-b, \lambda)\bigg).
	\end{equation}
	
\begin{definition}\label{D:Supersolution}	
We say that a function $m(f, F, \lambda)$ defined on $\Omega_{\bM}$ is a
\textit{supersolution} of the main inequality \eqref{E:MI-m} provided that
$m$ is non-negative, continuous, and satisfies 
	\begin{enumerate}[1)]
	\item The main inequality \eqref{E:MI-m};
	\item The obstacle condition
			$m(f, F, \lambda) = 1$, whenever $\lambda \leq F^2$.
	\end{enumerate}
\end{definition}

\begin{theorem} \label{T:bM-LstSup}
If $m$ is any supersolution as defined above, then $\bM \leq m$.
\end{theorem}

\begin{proof}
Obviously, it suffices to show that if $m$ is a supersolution, then
	\begin{equation}\label{E:bM-LstSup-Star}
	|J| m(f, F, \lambda) \geq |\{x \in J: S_J^2\varphi(x) \geq \lambda\}|,
	\end{equation}
for any function $\varphi$ supported in $J\in\cD$ with $\La\varphi\Ra_J = f$ and $\La|\varphi|\Ra_J = F$.

	\begin{remark} \label{R:ClassicHaar}
	Some caution is needed when working in $L^1$, so we recall here the \textit{classical Haar system} on $[0,1)$.
	Consider $J = [0,1)$ and arrange its dyadic subintervals (and hence also their corresponding Haar functions) in
	lexicographical order:
		$$J_n := \left[\frac{j-1}{2^k}, \frac{j}{2^k} \right), \:\:
		\forall n = 2^k + j - 1, \: k\geq 0, \: 1\leq j \leq 2^k. $$
	So
		$$ J_1 = J; \:\: J_2 = J_{-}, J_3 = J_{+}; \:\: J_{4} = J_{--}, J_5 = J_{-+}, \ldots, $$
	where $I_{-}$ and $I_{+}$ denote the left and right halves of a dyadic interval $I$, respectively.
	The classical result of Haar states that for every $\varphi \in L^p[0,1)$, $1\leq p <\infty$, the Haar series
		$$\varphi_N(x) := \unit_{[0,1)}(x)\int_0^1\varphi + \sum_{k=1}^N (\varphi, h_{J_{k}}) h_{J_{k}}(x) $$
	converges to $\varphi$ in $L^p[0,1)$ and almost everywhere.
	The reason for caution in our problem is that, while for $p>1$ the Haar functions form an unconditional basis for 
	$L^p[0,1)$, the most we can say for $p=1$ is that $\{\unit_{[0,1)}\}\cup\{h_{J_k}\}_{k\geq 1}$ is a Schauder basis.
	That is, we may rearrange the Haar series in such a way that it becomes divergent.
	
	This result transfers in an obvious way to any dyadic interval $J \in \cD$, and we use the notation
		$$ \{h_{J_k}\}_{k\geq 1} $$
	whenever we must keep track of the ordering of the subintervals of $J$. We say this is the Haar system adapted to $J$.
	\end{remark}

Returning to our proof, the first key observation is that it suffices to prove \eqref{E:bM-LstSup-Star} for functions $\varphi$ with 
\textit{finite Haar expansion}. To see this, let $\varphi$ with supp$(\varphi) \subset J$, $\La\varphi\Ra_J = f$ and $\La|\varphi|\Ra_J = F$. Then the Haar series
	$$\varphi_N := f \unit_J + \sum_{k=1}^N (\varphi, h_{J_k}) h_{J_k} $$
converges to $\varphi$ in $L^1(J)$ and almost everywhere. Moreover, $\La\varphi_N\Ra_J = f$ and $F_N := \La |\varphi_N|\Ra_J \rightarrow F$ as $N \rightarrow \infty$. Denote now the sets:
	$$E_{N,\lambda} := \{x\in J: S_J^2\varphi_N(x)\geq\lambda\}
	\:\:\:\text{ and }\:\:\: E_{\lambda} := \{x\in J: S_J^2\varphi(x) \geq \lambda\}.$$

We must be a little careful now, since it is not necessarily true that $|E_{N,\lambda}| \rightarrow |E_{\lambda}|$ as $N\rightarrow\infty$.
So let $\epsilon>0$ and use \eqref{E:bM-LstSup-Star} with $\lambda-\epsilon$ instead (and with $\varphi_N$ instead of $\varphi$), to obtain
	\begin{equation}
	\label{E:bM-LstSup-StarEps}
	|J| m(f, F_N, \lambda-\epsilon) \geq |\{x \in J: S_J^2\varphi_N(x) \geq \lambda-\epsilon\}| = |E_{N, \lambda-\epsilon}|,
	\end{equation}
for all $N$. Here we assumed  and used  \eqref{E:bM-LstSup-Star} for functions with 
\textit{finite Haar expansion}.
Now, since $S^2_J\varphi = \lim_{N\rightarrow\infty}S_J^2\varphi_N$ a.e. we have
	\begin{equation}
	\label{E:bM-LstSup-Star2}
	\bigcup_{N=1}^{\infty} E_{N,\lambda-\epsilon} \supseteq E_{\lambda} = \{x\in J: S_J^2\varphi(x) \geq\lambda\} \text{ a.e.}
	\end{equation}
(for almost all $x\in E_{\lambda}$, there will be a level $N_x$ after which $x\in E_{N,\lambda-\epsilon}$ for all $N \geq N_x$).
Taking $\limsup$ in \eqref{E:bM-LstSup-StarEps} and \eqref{E:bM-LstSup-Star2} we have then
	$$|J|m(f, F, \lambda-\epsilon) \geq \limsup_{N\rightarrow\infty} |E_{N,\lambda-\epsilon}| \geq |E_{\lambda}|.$$
Since this holds for all $\epsilon > 0$ and $m$ is continuous, we obtain exactly the desired conclusion  \eqref{E:bM-LstSup-Star} for general functions.

So now suppose supp$(\varphi) \subset J$, $\La\varphi\Ra_J = f$, $\La|\varphi|\Ra_J = F$. 
And also suppose that $\varphi$  has a finite Haar expansion. The goal is to show that if $m$ is any supersolution,
	$$|J| m(f, F, \lambda) \geq |E| \text{, where } E:= \{x\in J : S_J^2\varphi(x) \geq \lambda\}. $$
Suppose further that there is some dyadic level $N>0$ such that
	$$\varphi = f\unit_J + \sum_{\substack{I\subset J \\ |I| \geq |J| 2^{-N}}} (\varphi, h_I) h_I.$$
Remark that $S_J^2\varphi$ is constant on each $I \in J_{(N)}$, so $E$ is then a disjoint union of intervals $I \in J_{(N)}$
(unless $E$ is empty, in which case we are done).
For every $I \subset J$, let
	$$f_I := \La\varphi\Ra_I; \:\: F_I := \La |\varphi|\Ra_I; \:\: \lambda_I := \lambda - \sum_{K: I\subsetneq K\subset J} \Delta_K^2\varphi. $$
Then note that
	$$ f = f_J = \frac{1}{2}(f_{J_{+}} + f_{J_{-}}); \:\: 
		F = F_J = \frac{1}{2}(F_{J_{+}} + F_{J_{-}}); \:\: \lambda = \lambda_J;
		\:\: \Delta_J^2\varphi = \frac{1}{4}(f_{J_{+}} - f_{J_{-}})^2 \leq F^2_J. $$
Now, we describe the iteration procedure:
\begin{itemize}
\item If $\lambda \leq \Delta_J^2\varphi$, then the obstacle condition gives that $|J|m(f, F, \lambda) = |J| \geq |E|$, and we are done.
\item Otherwise, we have $\lambda_{J_{+}} = \lambda_{J_{-}} = \lambda - \Delta_J^2\varphi > 0$, so then we apply the main inequality for $m$ to obtain:
	$$ |J| m (f, F, \lambda) \geq |J_{-}| m(f_{J_{-}}, F_{J_{-}}, \lambda_{J_{-}}) + 
		|J_{+}| m(f_{J_{+}}, F_{J_{+}}, \lambda_{J_{+}}).$$
	\begin{itemize}
	\item If $\lambda_{J_{+}} \leq \Delta_{J_{+}}^2\varphi \leq F_{J_{+}}^2$, then this becomes
		$$ |J|m(f, F, \lambda) \geq |J_{-}|m(f_{J_{-}}, F_{J_{-}}, \lambda_{J_{-}}) + |J_{+}|,$$
	and if we iterate further, we only do so on $J_{-}$. Also note that, in this case, $\lambda_I \leq 0$ for any
	$I\in J_{(N)}$ with $I \subsetneq J_{+}$.
	\item Otherwise, iterate the $J_{+}$ term further, with $\lambda_{J_{+-}} = \lambda_{J_{++}} = \lambda - \Delta_J^2\varphi - \Delta_{J_{+}}^2\varphi > 0$.
	\end{itemize}
\end{itemize}

Continuing this process down to the last dyadic level $N$, we have
	\begin{equation}\label{E:bM-LstSup-2}
	|J|m(f, F, \lambda) \geq \sum_{I \in J_{(N)}: \lambda_I > 0} |I|m(f_I, F_I, \lambda_I) + 
		\sum_{I\in J_{(N)}: \lambda_I \leq 0} |I|.
	\end{equation}
Finally, it is easy to see that for any $I \in J_{(N)}$, we have $I \subset E$ if and only if $\lambda_I \leq \Delta_I^2\varphi$,
and again by the obstacle condition, if $I \subset E$ and $\lambda_I > 0$, then $m(f_I, F_I, \lambda_I) = 1$. 
So \eqref{E:bM-LstSup-2} gives us the desired conclusion:
	$$ |J|m(f, F, \lambda) \geq \sum_{I\in J_{(N)}: I \subset E} |I| = |E|. $$
\end{proof}

%-----------------------------%
%---------SubSection----------%
\subsection{$\bL$ is the Greatest Subsolution.} \label{Ss:bL-GrtSub}
Let us also consider the main inequality for $\bL$ in more generality:
	\begin{equation}\label{E:MI-l}
	\ell(f, p, \lambda+a^2) \leq \frac{1}{2}\bigg(\ell(f+a, p+b, \lambda) + \ell(f-a, p-b, \lambda)\bigg).
	\end{equation}

\begin{definition}\label{D:Subsolution}
We say that a function $\ell(f, p, \lambda)$ defined on $\Omega_{\bL}$ is a \textit{subsolution} 
for the main inequality \eqref{E:MI-l} provided that $\ell$ is non-negative, continuous, 
and satisfies 
	\begin{enumerate}[1).]
	\item The main inequality \eqref{E:MI-l};
	\item \textit{Range/Obstacle Condition:} $|f| \leq \ell(f, p, \lambda) \leq \max\{|f|, \sqrt{\lambda}\}$;
	\item \textit{Boundary Condition:} $\ell(f, 0, \lambda) = |f|$.
	\end{enumerate}
\end{definition}

\begin{theorem}\label{T:bL-GrtSub}
If $\ell$ is any subsolution as defined above, then $\ell \leq \bL$.
\end{theorem}

\begin{proof}
We must prove that $\ell(f, p, \lambda) \leq \La|\varphi|\Ra_J$ for any function $\varphi$ on $J$ with $\La\varphi\Ra_J = f$
and $\frac{1}{|J|}|E|= p$, where $E = \{x\in J : S_J^2\varphi(x) \geq \lambda\}$. As before, we may assume that there is some dyadic level $N \geq 0$ below which the Haar coefficients of $\varphi$ are zero, and assume that $p$ is a dyadic rational.

If $\lambda \leq \Delta_J^2\varphi$, then by condition 2):
	$$ \ell(f, p, \lambda) \leq \max\{|f|, \sqrt{\lambda}\} \leq \max\{|f|, |\Delta_J\varphi|\} \leq \La|\varphi|\Ra_J,$$
and we are done. Otherwise, put $\lambda_{J_{\pm}} = \lambda - \Delta_J^2\varphi > 0$, $f_{J_{\pm}} = \La\varphi\Ra_{J_{\pm}}$, and
	$$p_{J_{\pm}} = \frac{1}{|J_{\pm}|}|\{x \in J_{\pm}: S_{J_{\pm}}^2\varphi(x) \geq \lambda_{J_{\pm}}\}|.$$
Then by the Main Inequality:
	$$|J|\ell(f,p,\lambda) \leq |J_{-}|\ell(f_{J_{-}}, p_{J_{-}}, \lambda_{J_{-}}) +
		|J_{+}|\ell(f_{J_{+}}, p_{J_{+}}, \lambda_{J_{+}}).$$
If $\lambda_{J_{\pm}} \leq \Delta_{J_{\pm}}^2\varphi$, it follows as before that 
	$|J_{\pm}|\ell(f_{J_{\pm}}, p_{J_{\pm}}, \lambda_{J_{\pm}}) \leq \int_{J_{\pm}}|\varphi|$,
and otherwise we iterate further on $J_{\pm}$.

Continuing in this way down to the last level $N$ and putting 
	$\lambda_I := \lambda - \Delta_{I^{(1)}}^2\varphi - \ldots - \Delta_J^2\varphi$
for every $I \in J_{(N)}$, the previous iterations have covered all cases where $\lambda_I \leq 0$, and we have
	\begin{equation}\label{E:bL-GrtSub-temp} 
	|J|\ell(f, p, \lambda) \leq \sum_{I\in J_{(N)}: \lambda_I \leq 0} \int_I |\varphi| + 
		\sum_{I \in J_{(N)}: \lambda_I > 0} |I| \ell(f_I, p_I, \lambda_I).
	\end{equation}
Now note that for $I \in J_{(N)}$:
	$$p_I = \frac{1}{|I|}|\{x\in I: S_I^2\varphi(x) \geq \lambda_I\}| = 
		\frac{1}{|I|}|\{x\in I: \Delta_I^2\varphi(x) \geq \lambda_I\}| =
		\left\{ \begin{array}{ll}
		0 & \text{, if } I \not\subset E \\
		1 & \text{, if } I \subset E.
		\end{array}\right. $$
So, if $I \not\subset E$, then we use the boundary condition 3):
	$$\ell(f_I, p_I, \lambda_I) = \ell(f_I, 0, \lambda_I) = |f_I| \leq \La|\varphi|\Ra_I,$$
and if $I \subset E$, or $\lambda_I \leq \Delta_I^2\varphi$, we use condition 2) as before to obtain
	$\ell(f_I, p_I, \lambda_I) \leq \max\{|f_I|, |\Delta_I\varphi|\} \leq \La|\varphi|\Ra_I.$
Finally, \eqref{E:bL-GrtSub-temp} becomes:
	$$|J|\ell(f, p, \lambda) \leq \sum_{I \in J_{(N)}}\int_I |\varphi| = \int_J |\varphi|.$$
\end{proof}

\begin{remark} \label{R:GrtSub-Bdry}
Later in Section \ref{S:Boundaries}, we will look at subsolutions for the particular case $\bL(f, 1, \lambda)$.
We note that the boundary condition 3). above will no longer be needed there: when $p = 1$, we are looking only at functions
$\varphi$ with $S_J^2\varphi \geq \lambda$ \textit{almost everywhere} on $J$, so at the end of the proof,
there will be no intervals left outside $E$, and there will be no terms of the form $\ell(f_I, 0, \lambda_I)$.
\end{remark}

\begin{remark} \label{R:BollDef}
Our definition of the Bellman function $\bL$ was inspired by Bollobas \cite{Bollobas}, who worked with
$$ L_B(s, h) := \inf \left\{\int_0^1|\varphi|\,dx : \text{supp}(\varphi) \subset [0,1]; \: \int_0^1\varphi\,dx = h; \:
	S\varphi \equiv s \text{ on } [0,1] \right\}. $$
We claim that $L_B(s, h) = \bL(h, 1, s^2)$. In fact, we may define $\bL(f, p, \lambda)$ in general by replacing ``$\geq \lambda$'' with
``$= \lambda$.'' To see this, let
	$$\bL'(f, p, \lambda) := \inf\{\La|\varphi|\Ra_J : supp(\varphi) \subset J; \: \La \varphi\Ra_J = f; \:
		\frac{1}{|J|}|\{x\in J: S_J^2\varphi(x) = \lambda\}| = p\}.$$
We claim that $\bL' = \bL$. Suppose $\varphi$ is admissible for $\bL'(f, p, \lambda)$. Then 
	$$q := \frac{1}{|J|}|\{x\in J: S_J^2\varphi(x) \geq \lambda\}| \geq \frac{1}{|J|}|\{x\in J: S_J^2\varphi(x) = \lambda\}| = p,$$
so $\varphi$ is also admissible for $\bL(f, q, \lambda)$ with $q \geq p$. Then, since $\bL$ is non-decreasing in the second variable,
$\La|\varphi|\Ra_J \geq \bL(f, q, \lambda) \geq \bL(f, p, \lambda)$. This shows that $\bL' \geq \bL$. 

\medskip

To see the converse, we note that $\bL'$ is a subsolution for the main inequality \eqref{E:MI-l}, as in Definition \ref{D:Subsolution}.
 It is easy to show in the usual way that $\bL'$ satisfies \eqref{E:MI-l}.
Moreover, $\bL'$ satisfies the same range condition \eqref{E:Range-bL} as $\bL$:
	$|f| \leq \bL'(f, p, \lambda) \leq (p-1)|f| + p \max(|f|, \sqrt{\lambda}).$
The proof of this inequality for $\bL$ goes through identically for $\bL'$, since the test functions $\varphi$ we constructed for each dyadic rational $p$ really satisfied $\{x\in J: S_J^2\varphi(x) \geq \lambda\} = \{x\in J: S_J^2\varphi(x) = \lambda\}$.
Then by Theorem \ref{T:bL-GrtSub} (that claims $\bL$ to be the greatest subsolution for the main inequality \eqref{E:MI-l}) it follows that $\bL' \leq \bL$.
\end{remark}

%--------------------------%
%---------Section----------%
%--------------------------%
\section{Relationships between $\bM$ and $\bL$} \label{S:RelsML}

\begin{theorem} \label{T:MLRel}
$\bL(f, p, \lambda)$ is the smallest value of $F$ for which $\bM(f, F, \lambda) \geq p$:
	\begin{equation} \label{E:ML-rel}
	\bL(f, p, \lambda) = \inf\{F \geq |f|: \bM(f, F, \lambda) \geq p\}.
	\end{equation}
Moreover, $\bM(f, F, \lambda)$ is the largest value of $p$ such that $\bL(f, p, \lambda) \leq F$:
	\begin{equation} \label{E:LM-rel}
	\bM(f, F, \lambda) = \sup\{p\in [0, 1]: \bL(f, p, \lambda) \leq F\}.
	\end{equation}
\end{theorem}

%Remark that it also follows from the above that:
	%\begin{equation}\label{E:ML-Equals}
	%\bM\big(f, \bL(f, p, \lambda), \lambda \big) = p \:\:\:\text{ and }\:\:\:
	%\bL\big( f, \bM(f, F, \lambda), \lambda) = F.
	%\end{equation}

\begin{proof}
Suppose $\bM(f, F, \lambda) \geq p$ and let $\epsilon > 0$. Then there is a function $\varphi$ on $J \in \cD$ such that:
	$$ \La\varphi\Ra_J = f, \:\: \La|\varphi|\Ra_J = F, \:\: 
	q:= \frac{1}{|J|} |\{x\in J: S_J^2\varphi(x) \geq \lambda\}| > p - \epsilon.$$
Then $\varphi$ is admissible for $\bL(f, q, \lambda)$, and since $\bL$ is non-decreasing in the second variable,
	$$\bL(f, p-\epsilon, \lambda) \leq \bL(f, q, \lambda) \leq \La|\varphi|\Ra_J = F.$$
Since this holds for all $\epsilon > 0$,
	$ \bL(f, p, \lambda) \leq F$ for all $F$  such that $\bM(f, F, \lambda)\geq p. $
Further, for every $\epsilon > 0$ there is a function $\varphi$ on $J \in \cD$ such that
	$$ \La\varphi\Ra_J = f, \:\:  \frac{1}{|J|} |\{x\in J: S_J^2\varphi(x) \geq \lambda\}| = p, \:\:
		F:= \La|\varphi|\Ra_J < \bL(f, p, \lambda) + \epsilon.$$
But $\varphi$ is admissible for $\bM(f, F, \lambda)$, and then clearly $\bM(f, F, \lambda) \geq p$.
This proves \eqref{E:ML-rel}.
The other equation \eqref{E:LM-rel} follows similarly.
\end{proof}

%-----------------------------%
%---------SubSection----------%
\subsection{Optimal Obstacle Conditions for $\bM$ and $\bL$.} \label{Ss:OptimalOC}

Looking back at the obstacle condition \eqref{E:OC-bM} for $\bM$, namely $\bM(f, F, \lambda) = 1$ whenever $F \geq \sqrt{\lambda}$,
there is no reason to think this condition is optimal. That is, there well could be values of $F$ strictly smaller than $\sqrt{\lambda}$ where $\bM$ is $1$. As it turns out, the optimal obstacle condition for $\bM$ can be obtained from information about $\bL$. 
Since $\bM \leq 1$, taking $p = 1$ in \eqref{E:ML-rel}, we obtain exactly this:
	\begin{equation}\label{E:bM-SharpOC-L}
	\bL(f, 1, \lambda) = \inf\{F\geq |f|: \bM(f, F, \lambda) = 1\}.
	\end{equation}
On the other hand, the obstacle condition for $\bL$ really comes from its range,
	$|f| \leq \bL(f, p, \lambda) \leq \max\{|f|, \sqrt{\lambda}\}$,
 which clearly shows that $\bL = |f|$ whenever $|f| \geq \sqrt{\lambda}$. However, this says nothing about $p$, and we do know that, for example,
$\bL(f, 0, \lambda) = |f|$ regardless of the behavior of $f$ and $\lambda$. What other values of $p$ could this hold for? This is again obtained precisely from information about $\bM$, by letting $F = |f|$ in \eqref{E:LM-rel}:
	\begin{equation} \label{E:bL-SharpOC-M}
	\bM(f, |f|, \lambda) = \sup\{p\in[0,1]: \bL(f, p, \lambda) = |f|\}.
	\end{equation}
So, if we find the expressions for $\bL$ and $\bM$ along these boundaries of their domains, we also obtain the optimal obstacle conditions for $\bM$ and $\bL$, respectively. 

We denote these boundary values of $\bM$ and $\bL$ by $\bM_b$ and $\bL_b$, respectively, defined as follows.
For $f \geq 0$ and $\lambda > 0$,
	\begin{equation}\label{E:Mb-Def}
	\bM_b(f, \lambda) := \bM(f, |f|, \lambda) = \sup \frac{1}{|J|} |\{x\in J: S_J^2\varphi(x) \geq \lambda\}|, 
	\end{equation}
where the supremum is over all functions $\varphi$ on $J$ with $\varphi \geq 0$ a.e. and $\La\varphi\Ra_J = f$.
Note that since $\bM$ is even in $f$, it suffices to consider $\bM_b$ for $f \geq 0$. Moreover, the only admissible functions for 
$\bM(f, |f|, \lambda)$ are those with $\varphi \geq 0$ a.e. (for $f \geq 0$) or $\varphi \leq 0$ a.e. (for $f \leq 0$).
Similarly,
	\begin{equation}\label{E:Lb-Def}
	\bL_b(f, \lambda) := \bL(f, 1, \lambda) = \inf\{\La|\varphi|\Ra_J: \text{supp}(\varphi) \subset J; \: \La \varphi\Ra_J = f; \:
		S_J^2\varphi \geq \lambda \text{ a. e. on } J\}.
	\end{equation}
We find these functions in Section \ref{S:Boundaries}, where we prove the following results.

\begin{theorem}\label{T:Mb-Bdry}
The function $\bM_b$ is given by
	\begin{equation}\label{E:bM-Bdry}
	\bM_b(|f|, \lambda) = \bM(f, |f|, \lambda) = \left\{ \begin{array}{ll}
		\frac{\Phi\left(\frac{|f|}{\sqrt{\lambda}}\right)}{\Phi(1)}, & |f| < \sqrt{\lambda} \\
		1, & |f| \geq \sqrt{\lambda}.
	\end{array}\right. = \min\left(\frac{\Phi(|f|/\sqrt{\lambda})}{\Phi(1)}, 1\right),
	\end{equation}
where
	$$\Phi(\tau) := \int_0^{\tau} e^{-x^2/2}\,dx,$$
for all $\tau \geq 0$.
\end{theorem}

\begin{theorem}\label{T:Lb-Bdry}
The function $\bL_b$ is given by
	\begin{equation}\label{E:bL-Bdry}
	\bL_b(f, \lambda) = \bL(f, 1, \lambda) = \left\{ \begin{array}{ll}
		\frac{\sqrt{\lambda}\Psi\left(\frac{|f|}{\sqrt{\lambda}}\right)}{\Psi(1)}, & 0 \leq |f| < \sqrt{\lambda}\\
		|f|, & |f| \geq \sqrt{\lambda}.
		\end{array}\right. = \sqrt{\lambda} \max\left( \frac{\Psi(|f|/\sqrt{\lambda})}{\Psi(1)}, 
		\frac{|f|}{\sqrt{\lambda}} \right),
	\end{equation}
where
	$$\Psi(\tau) = \tau \Phi(\tau) + e^{-\tau^2/2},$$
for all $\tau \geq 0$.
\end{theorem}

%-----------------------------%
%---------SubSection----------%
\subsection{The functions $\theta$ and $\eta$.}
To visualize the optimal obstacle conditions induced by $\bM_b$ and $\bL_b$ for $\bL$ and $\bM$, respectively, we use homogeneity of $\bM$ and $\bL$ to reduce the discussion to functions of two variables. 
Specifically, from \eqref{E:Hom-bM} and \eqref{E:Hom-bL}, we write
	\begin{equation}\label{E:Def-te}
	\bM(f, F, \lambda) = \bM(f/\sqrt{\lambda}, F/\sqrt{\lambda}, 1) =: \theta(\tau, \gamma) \:\:\text{ and }\:\:
	\bL(f, p, \lambda) =: \sqrt{\lambda} \eta(\tau, p),
	\end{equation}
where $\tau = f/\sqrt{\lambda}$ and $\gamma = F/\sqrt{\lambda}$. 
Thus $\theta$ is defined on $\Omega_{\theta} := \{0 \leq |\tau| \leq \gamma\}$ with values in $[0,1]$, 
and $\eta$ is defined on 
$\Omega_{\eta} := \{0 \leq p \leq 1; \: \tau \in \mathbb{R}\}$
with values satisfying $|\tau| \leq \eta(\tau, p) \leq (1-p)|\tau| + p\max(|\tau|, 1)$.
It is also clear that $\theta$ and $\eta$ are even in $\tau$, so we often restrict our attention to the domains 
$\Omega_{\theta}^{+}$ and $\Omega_{\eta}^{+}$ where $\tau \geq 0$.
Other properties that $\theta$ and $\eta$ inherit from $\bM$ and $\bL$ are easy to check:
	\begin{itemize}
	\item $\theta(0,0) = 0$ and $\eta(\tau, 0) = \tau$.
	\item $\theta$ is maximal at $\tau = 0$, and $\eta$ is minimal at $\tau = 0$:
		$$\theta(|\tau|, \gamma) \leq \theta(0, \gamma); \:\: \eta(0,p) \leq \eta(|\tau|, p).$$
	\item $\theta$ is decreasing in $\tau$ for $\tau \geq 0$, and is increasing in $\gamma$.
		$\eta$ is increasing in both $\tau \geq 0$ and $p$.
	\item $\theta$ is concave in both $\tau$ and $\gamma$, and $\eta$ is convex in both $\tau$ and $p$.
	\item The original obstacle conditions \eqref{E:OC-bM} and \eqref{E:OC-bL} for $\bM$ and $\bL$ translate to
	$$\theta(\tau, \gamma) = 1, \: \forall \gamma \geq 1 \:\:\text{ and }\:\:
	\eta(\tau, p) = |\tau|, \: \forall |\tau| \geq 1.$$

	\end{itemize}

Moreover, \eqref{E:bM-SharpOC-L} and \eqref{E:bL-SharpOC-M} become
	$$\eta(\tau, 1) = \inf\{\gamma\geq|\tau|: \theta(\tau, \gamma) = 1\} \:\:\text{ and }\:\:
		\theta(\tau, |\tau|) = \sup\{p: \eta(|\tau|, p) = |\tau|\}.$$
The expression for $\bL_b$ gives that
	$$\eta(\tau, 1) = \begin{cases}
		\frac{\Psi(|\tau|)}{\Psi(1)}, & 0 \leq |\tau| < 1,\\
		|\tau|, & |\tau| \geq 1\,.
		\end{cases}
	$$
	That yields the optimal obstacle condition for $\theta$ (see Figure \ref{fig1}).
Similarly, $\bM_b$ gives that
	$$\theta(\tau, \tau) = \begin{cases}
	 \frac{\Phi(|\tau|)}{\Phi(1)}, & 0 \leq |\tau| \leq 1,\\
	1,  & \tau \geq 1\,.
	\end{cases}
	$$
And that yields the optimal obstacle condition for $\eta$ (see Figure \ref{fig2}).

\begin{figure}[h]
\centering
\begin{minipage}{.5\textwidth}
  \centering
  \includegraphics[scale=0.32]{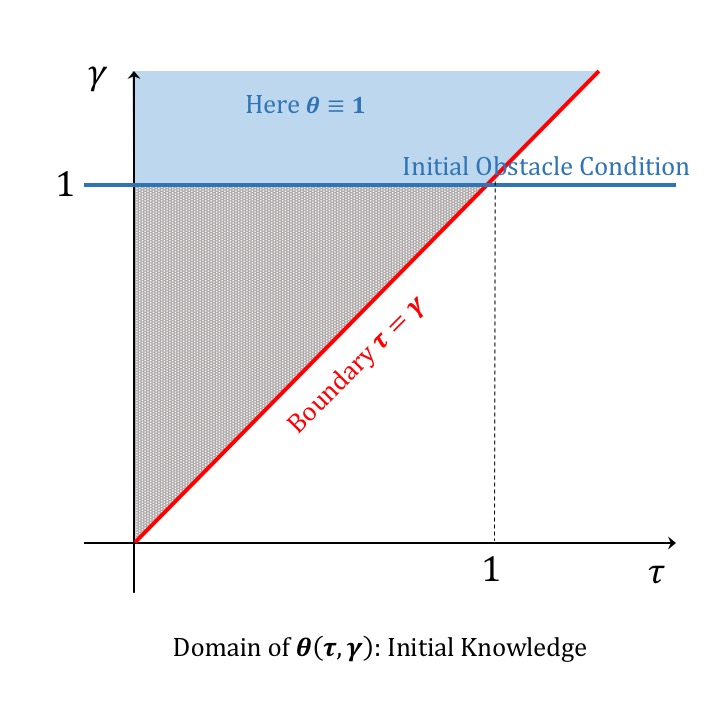}
\end{minipage}%
\begin{minipage}{.5\textwidth}
  \centering
  \includegraphics[scale=0.32]{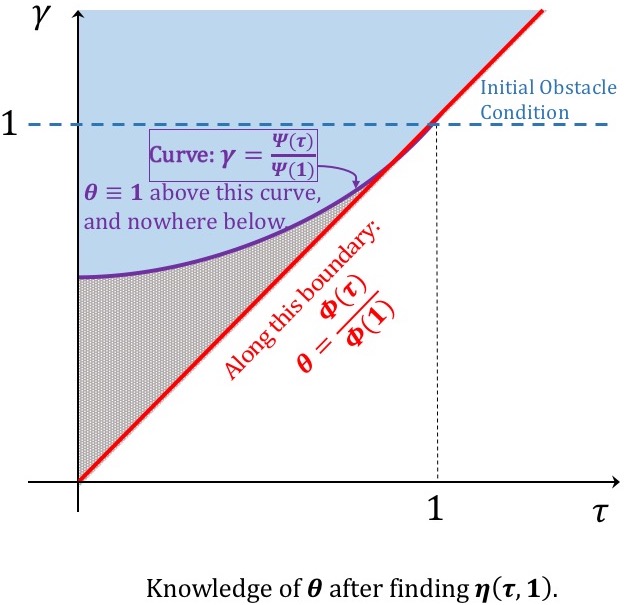}
\end{minipage}
\caption{Initial and optimal Obstacle Conditions for $\theta$.}
\label{fig1}
\end{figure}

\begin{figure}[h]
\centering
\begin{minipage}{.5\textwidth}
  \centering
  \includegraphics[scale=0.32]{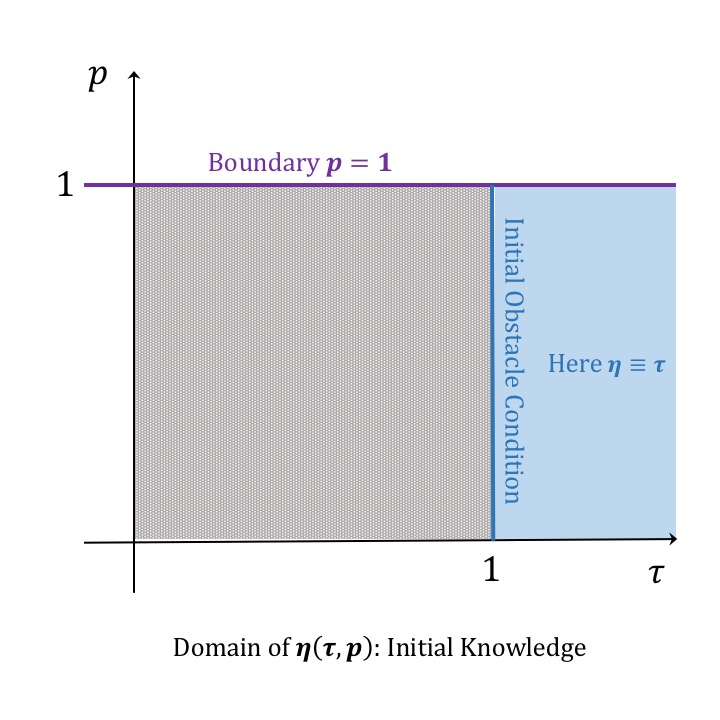}
\end{minipage}%
\begin{minipage}{.5\textwidth}
  \centering
  \includegraphics[scale=0.32]{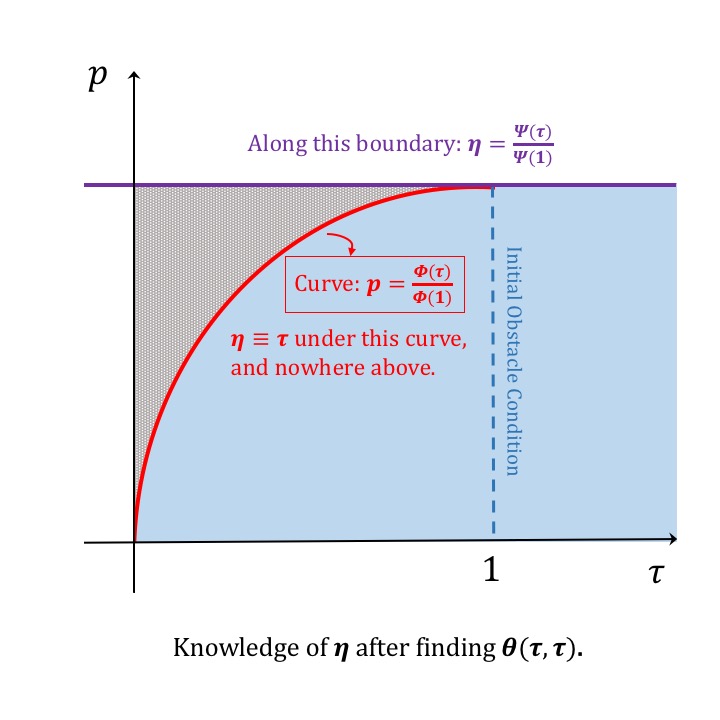}
\end{minipage}
\caption{Initial and optimal Obstacle Conditions for $\eta$.}
\label{fig2}
\end{figure}

Let us give some special names to the ``interesting'' parts of the domains of $\theta$ and $\eta$, where they are unknown.
We denote by $\widetilde{\Omega_{\theta}}$ the part of the domain of $\theta$ that lies underneath the obstacle condition curve 
$\gamma = \eta(\tau, 1)$:
	$$\widetilde{\Omega_{\theta}} := \{(\tau, \gamma): 0 \leq |\tau|\leq 1; \:
	|\tau| \leq \gamma \leq \frac{\Psi(|\tau|)}{\Psi(1)} = \eta(\tau, 1)\},$$
and by $\widetilde{\Omega_{\eta}}$ the part of the domain of $\eta$ that lies above the obstacle condition curve $p = \theta(\tau, |\tau|)$:
	$$\widetilde{\Omega_{\eta}} := \{ (\tau, p): 0 \leq |\tau|\leq 1; \: p \geq \frac{\Phi(|\tau|)}{\Phi(1)}=\theta(\tau, |\tau|) \}.$$
As the next proposition shows, in these domains we can improve the results of Theorem \ref{T:MLRel}.

\begin{prop} \label{P:ML-Equals}
The functions $\bM$ and $\bL$ satisfy
	\begin{equation}\label{E:ML-Equals}
	\bM\big(f, \bL(f, p, \lambda), \lambda \big) = p \text{, for all } (\tau = f/\sqrt{\lambda}, \: p) \in \widetilde{\Omega_{\eta}},
	\end{equation}
that is, for all
	$$0 \leq |f| \leq \sqrt{\lambda} \text{ and } p \geq \frac{\Phi(|f|/\sqrt{\lambda})}{\Phi(1)}.$$
Similarly,
	\begin{equation}\label{E:LM-Equals}
	\bL\big( f, \bM(f, F, \lambda), \lambda) = F \text{, for all } (\tau = f/\sqrt{\lambda}, \: \gamma=F/\sqrt{\lambda}) \in
	\widetilde{\Omega_{\theta}},
	\end{equation}
that is, for all
	$$0 \leq |f| \leq \sqrt{\lambda} \text{ and } |f| \leq F \leq \frac{\sqrt{\lambda}\Psi(|f|/\sqrt{\lambda})}{\Psi(1)}.$$
\end{prop}

\begin{proof}
The relationships between $\bM$ and $\bL$ in Theorem \ref{T:MLRel} translate in $\theta$--$\eta$ language to
	\begin{equation}\label{E:et}
	\eta(\tau, p) = \inf\{\gamma \geq \tau: \: \theta(\tau, \gamma) \geq p\} \:\:\text{ and }\:\:
	\theta(\tau, \gamma) = \sup\{0 \leq p \leq 1: \eta(\tau, p) \leq \gamma\}.
	\end{equation}
Now fix some $0 \leq \tau \leq 1$. If $p < \theta(\tau, \tau)$ (below the obstacle condition curve for $\eta$), then
$\eta(\tau, p) = \tau$ and $\theta(\tau, \gamma) \geq \theta(\tau, \tau) > p$ for all $\gamma \geq \tau$,
so indeed $\gamma = \tau$ is the smallest possible value of $\gamma$ where $\theta(\tau, \gamma) \geq p$.
If, on the other hand, $1 \geq p \geq \theta(\tau, \tau)$, or $(\tau, p) \in \widetilde{\Omega_{\eta}}$, 
then there exists a $\gamma \geq \tau$ such that
$\theta(\tau, \gamma) = p$. So, in this case, we may rewrite the first equation in \eqref{E:et} as
	$$\eta(\tau, p) = \inf\{\gamma \geq \tau: \theta(\tau, \gamma) = p\},$$
and then obviously
	\begin{equation}\label{E:te-Equals}
	\theta(\tau, \eta(\tau, p)) = p \text{, for all } (\tau, p) \in \widetilde{\Omega_{\eta}}.
	\end{equation}
This is exactly \eqref{E:ML-Equals}. Similarly, we have that
	\begin{equation}\label{E:et-Equals}
	\eta(\tau, \theta(\tau, \gamma)) = \gamma \text{ for all } (\tau, \gamma) \in \widetilde{\Omega_{\theta}}.
	\end{equation}
\end{proof}

%--------------------------%
%---------Section----------%
%--------------------------%
\section{The Sharp Inequality for The Square Function}
\label{S:Sharp}

The following result is an adaptation of Lemma 2 in Bollobas \cite{Bollobas}.

\begin{theorem}\label{T:Ineq-Bol}
The functions $\bM$ and $\bL$ satisfy:
	\begin{equation}\label{E:Ineq-Bol}
	\bM(0, F, \lambda) \leq \frac{F}{\bL(0, 1, \lambda)} = \frac{F}{\bL_b(0, \lambda)},
	\end{equation}
for all $F\geq 0$ and $\lambda > 0$.
\end{theorem}

\begin{proof}
Let $\varphi$ be a function on $J \in \cD$ with $\int_J\varphi = 0$ and finite Haar expansion (up to some dyadic level $N \geq 0$):
	$$\varphi = \sum_{I\subset J} (\varphi, h_I) h_I = \sum_{k=1}^{2^{N+1}-1} a_k h_{J_{k}},$$
where in the last term we are keeping track of the ordering in the Haar system adapted to $J$, as in Remark \ref{R:ClassicHaar}.
Fix some $\lambda > 0$ and let
	$$p:= \frac{1}{|J|}|\{x \in J: S_J^2\varphi(x) \geq \lambda\}| =: c\La|\varphi|\Ra_J,$$
and suppose that $0<p<1$. Put the intervals in the last generation $J_{(N)}$ into two (``good'' and ``bad'') categories: 
	$$J_{(N)} = \cI_g \cup \cI_b,$$
where $\cI_g$ is the collection of intervals $I \in J_{(N)}$ with $S_J^2\varphi \geq \lambda$ on $I$, and $\cI_b$ are the remaining ones where $S_J^2\varphi < \lambda$. Then clearly
	$$\left| \cup_{I\in\cI_g} I\right| = p|J| \text{ and } \left| \cup_{I\in\cI_b} I\right| = (1-p)|J|.$$

Now, for each $I \in \cI_b$, let the function:
	$$\psi_I := \sum_{k=1}^{2^{N+1}-1} \frac{1}{\sqrt{2^{N+1}}} a_k h_{I_k^{-}} +
		\sum_{k=1}^{2^{N+1}-1} \frac{1}{\sqrt{2^{N+1}}} a_k h_{I_k^{+}},$$
where each $\{h_{I_k^{-}}\}$ and $\{h_{I_k^{+}}\}$ denote the (ordered) Haar systems adapted to $I_{-}$ and $I_{+}$, respectively.
Essentially, this amounts to
	$$\psi_I = \unit_{I_{-}}\psi_{I_{-}} + \unit_{I_{+}}\psi_{I_{+}},$$
where each $\psi_{I_{\pm}}$ is a \textit{copy of $\varphi$ adapted to $I_{\pm}$}, so 
	$$\La|\psi_{I_{\pm}}|\Ra_{I_{\pm}} = \La|\psi_I|\Ra_I = \La|\varphi|\Ra_J.$$
Now, let
	$$\varphi_1 := \varphi + \sum_{I \in \cI_b} \psi_I.$$
Then $\int_J\varphi_1 = 0$, and
	$$\La|\varphi_1|\Ra_J \leq \La|\varphi|\Ra_J \big(1 + (1-p)\big).$$
The square function $S_J^2\varphi_1$ equals $S_J^2\varphi$ on $\cup_{I\in\cI_g}I$, while on any $I \in \cI_b$:
	$$
	|\{x\in I : S_J^2\varphi_1(x) \geq \lambda\}| \geq |I_{-}|p + |I_{+}|p = |I|p.
	$$
So $\varphi_1$ satisfies
	$$
	\frac{1}{|J|}|\{x \in J: S_J^2\varphi_1(x) \geq \lambda\}| \geq p\big(1+ (1-p)\big).
	$$
Continuing this process, we obtain a sequence of functions $\{\varphi_n\}_n$, supported on $J$, each with $\int_J\varphi_n = 0$ and
	$$\frac{1}{|J|}|\{x\in J: S_J^2\varphi_n(x) \geq \lambda\}| \geq p\big(1 + (1-p) + \ldots + (1-p)^n\big) 
	\xrightarrow[n\rightarrow\infty]{} 1,$$	
and
	$$\La|\varphi_n|\Ra_J \leq \La|\varphi|\Ra_J \big(1 + (1-p) + \ldots + (1-p)^n\big) 
	\xrightarrow[n\rightarrow\infty]{} \frac{1}{p} \La|\varphi|\Ra_J.$$
Letting $\widetilde{\varphi} = \lim \varphi_n$ in $L^1$, we have
	$$\La\widetilde{\varphi}\Ra_J = 0, \:\: \La|\widetilde{\varphi}|\Ra_J \leq \frac{1}{p}\La|\varphi|\Ra_J, \:\:
		S_J^2\widetilde{\varphi} \geq \lambda \text{ a.e. on } J. $$
Therefore $\widetilde{\varphi}$ is admissible for $\bL(0,1,\lambda)$, so
	$$ \bL(0,1,\lambda) \leq  \La|\widetilde{\varphi}|\Ra_J \leq \frac{1}{p}\La|\varphi|\Ra_J = \frac{1}{c}.$$
We then have that
	$$\frac{1}{|J|}|\{x\in J: S_J^2\varphi(x) \geq\lambda\}| \leq \frac{\La|\varphi|\Ra_J}{\bL(0,1,\lambda)},$$
for all $\varphi$ on $J$ with mean zero, and all $\lambda > 0$, which yields exactly \eqref{E:Ineq-Bol}.
\end{proof}

Next, we find the values of $\bM$ and $\bL$ for $f = 0$.

\begin{prop}\label{P:Bdryf0}
If $f = 0$,  the functions $\bM$ and $\bL$ are given by:
	\begin{equation} \label{E:Bdryf0-bM}
	\bM(0, F, \lambda) = \left\{ \begin{array}{ll}
		\frac{F}{\bL(0,1,\lambda)} = \frac{F}{\sqrt{\lambda}}\Psi(1), & \text{ if } F \leq \frac{\sqrt{\lambda}}{\Psi(1)}\\
		1, & \text{ if } F > \frac{\sqrt{\lambda}}{\Psi(1)}.
	\end{array}\right.
	\end{equation}
and
	\begin{equation} \label{E:Bdryf0-bL}
	\bL(0, p, \lambda) = p \: \bL(0,1,\lambda) = \frac{p\sqrt{\lambda}}{\Psi(1)}.
	\end{equation}
\end{prop}

\begin{proof}
Consider $\gamma \mapsto \theta(0, \gamma)$. We know that $\theta(0,0)=0$ and 
$\theta(0,\gamma)=1$ for all $\gamma \geq \frac{1}{\Psi(1)}$ (see Figure \ref{fig1}).
But $\theta$ is \textit{concave} in $\gamma$, so $\theta(0, \cdot)$ lies above its secant line
between $(0,0,0)$ and $(0,\frac{1}{\Psi(1)}, 1)$. This line has equation $y(\gamma) = \Psi(1)\gamma$,
so
	$$\theta(0,\gamma) \geq \Psi(1)\gamma \text{, for all } 0 \leq \gamma \leq \frac{1}{\Psi(1)}.$$
But Theorem \ref{T:Ineq-Bol} says that $\theta(0, \gamma) \leq \Psi(1)\gamma$, so then
	\begin{equation*}
	\theta(0, \gamma) = \left\{\begin{array}{ll}
		\Psi(1)\gamma, & \text{ if } 0 \leq \gamma \leq \frac{1}{\Psi(1)}\\
		1, & \text{ otherwise.}
	\end{array}\right.
	\end{equation*}
Now let $p \in [0,1]$. Then $p = \theta(0, \gamma)$ for $\gamma = \frac{p}{\Psi(1)}$. Then by \eqref{E:et-Equals},
	$$\eta(0,p) = \eta(0, \theta(0, \gamma)) = \gamma = \frac{p}{\Psi(1)},$$
proving that
	$$\eta(0,p) = \frac{p}{\Psi(1)}.$$
\end{proof}

\begin{corollary}\label{C:SharpC}
The sharp constant $C$ in the inequality
	$$\frac{1}{|J|}|\{x\in J: S_J^2\varphi(x) \geq \lambda\}| \leq C \frac{1}{\sqrt{\lambda}}\La|\varphi|\Ra_J
	\text{, for all } \varphi \in L^1(J), \: J \in \cD,$$
is given by $C = \Psi(1)$.
\end{corollary}

\begin{proof}
Obviously
	\begin{equation}\label{E:Sharp-C}
	C = \sup_{f, F, \lambda} \frac{\bM(f, F, \lambda) \sqrt{\lambda}}{F} = 
	\sup_{F, \lambda} \frac{\bM(0, F, \lambda) \sqrt{\lambda}}{F} = \Psi(1),
	\end{equation}
where the second equality follows since $\bM(f, F, \lambda) \leq \bM(0, F, \lambda)$,
and the last equality follows from \eqref{E:Bdryf0-bM}.
\end{proof}

%--------------------------%
%---------Section----------%
%--------------------------%
\section{Proofs of the Boundary Values $\bM_b$ and $\bL_b$ of $\bM$ and $\bL$}\label{S:Boundaries}

In this section we prove Theorems \ref{T:Mb-Bdry} and \ref{T:Lb-Bdry}.

\subsection{The boundary case $\bM_b(f, \lambda)$.}

Recall that
	$$\bM_b(f, \lambda) := \sup \frac{1}{|J|} |\{x\in J: S_J^2\varphi(x) \geq \lambda\}|, \forall f \geq 0, \lambda > 0,$$
where the supremum is over all functions $\varphi$ on $J$ with $\varphi \geq 0$ a.e. and $\La\varphi\Ra_J = f$.
Then $\bM_b$ has the obvious properties:
	\begin{itemize}
	\item Domain: $\Omega^{+}_{\bM_b} = \{f \geq 0, \lambda > 0\}$; Range: $0 \leq \bM_b \leq 1$;
	\item $\bM_b$ is decreasing in $\lambda$;
	\item Homogeneity: $\bM_b(f, \lambda) = \bM_b(tf, t^2\lambda)$, for all $t>0$;
	\item Obstacle Condition: $\bM_b(f, \lambda) = 1$, for all $f \geq \sqrt{\lambda}$;
	\item Boundary Condition:	$\bM_b(0, \lambda) = 0$, for all $\lambda > 0$;
	\item Main Inequality: For any pairs in the domain with $f = \frac{1}{2}(f_{+}+f_{-})$, $\lambda = \min\{\lambda_{\pm}\}$:
		\begin{equation}\label{E:M-MI}
		\bM_b\bigg(f, \lambda + \left(\frac{f_{+}-f_{-}}{2}\right)^2\bigg) \geq \frac{1}{2}
			\big(\bM_b(f_{+}, \lambda_{+}) + \bM_b(f_{-}, \lambda_{-}) \big);
		\end{equation}
	\item $\bM_b$ is concave and non-decreasing in $f$;
	\item Least Supersolution: If $m(f, \lambda)$ is a continuous non-negative function on $\Omega^{+}_{\bM_b}$ which satisfies
		\eqref{E:M-MI} and the obstacle condition, then $\bM_b \leq m$. 
	\end{itemize}

Rewrite the Main Inequality \eqref{E:M-MI} in a more convenient form:
	\begin{equation}
	\label{discrM}
	\bM_b(f, \lambda) \geq \frac{1}{2}\big( \bM_b(f-a, \lambda - a^2) + \bM_b(f+a, \lambda - a^2)\big), 
		\forall f\geq a \geq 0, \lambda > a^2, 
		\end{equation}
		Using homogeneity of $\bM_b$, we put:
	$$
	\bM_b(f, \lambda) = \bM_b(f/\sqrt{\lambda}, 1) =: \alpha(\tau) \text{, where } \tau := \frac{f}{\sqrt{\lambda}}.
	$$
	Then from \eqref{E:bM-F0},
	$$
	\alpha : [0, \infty) \rightarrow [0, 1] \text{ with }
	 \alpha(0) = 0 \text{ and } \alpha(\tau) = 1, \forall \tau \geq 1. 
	$$

	We can rewrite inequality  \eqref{discrM} in terms of $\alpha$  as follows ($h:= \frac{a}{\sqrt{\la}}$, $ \tau := \frac{f}{\sqrt{\lambda}}$):
	\begin{equation}
	\label{discrALPHA}
	\frac12\Big[ \alpha\Big(\frac{\tau - h}{\sqrt{1-h^2}}\Big) + \alpha\Big(\frac{\tau + h}{\sqrt{1-h^2}}\Big) \Big] - \alpha(\tau) \le 0\,.
	\end{equation}	
	Since $\bM_b$ is concave in the first variable, we know that $\alpha$ is concave.
	We will now use the second order   a.e. Taylor formula for concave functions
	from \cite{EG}:
\begin{equation}
\label{concA}
F(\tau+\eps) = F(\tau) + F'(\tau) \eps + \frac12 F''(\tau) \eps^2 + 
o(\eps^2), \quad \eps\to 0, \quad \text{for a.e.}\,\,x\,,
\end{equation}
We use this formula in conjunction with \eqref{discrALPHA}. We also use the  expansions:
$$
\frac{\tau \pm h}{\sqrt{1-h^2}} - \tau= \pm h +\frac12 \tau  h^2 + o(h^2)=:\eps\,.
$$

The inequality \eqref{discrALPHA} then obviously implies the following inequality valid a.e.:
	\begin{equation}\label{E:alpha-DI}
	\tau \alpha'(\tau) + \alpha''(\tau) \leq 0.
	\end{equation}
	
	 But function $\alpha$ is concave. 
In particular, it is everywhere defined and continuous, and its derivative $\alpha'$ is precisely its distributional derivative, and it is everywhere defined
decreasing function.  Let $(\alpha)''$ denote the distributional derivative of decreasing function $\alpha'$. Thus it is a non-positive measure. We denote its singular part by symbol $\sigma_s$. 
Hence, in the sense of distributions
\begin{equation}\label{E:alpha-DI-distr}
	\tau \alpha'd\tau + (\alpha)'' = (\tau \alpha(\tau) + \alpha''(\tau) )d\tau + \sigma_s \leq 0.
	\end{equation}

Let us look at the differential equation $\tau y'(\tau) + y''(\tau) = 0$ for $\tau \geq 0$. The general solution is:
	$$y(\tau) = C\Phi(\tau) + D \text{, where } \Phi(\tau) := \int_{0}^{\tau} e^{-x^2/2}\,dx.$$
Imposing $y(0) = 0$ and $y(1) = 1$, we obtain an obvious candidate for our function $\alpha$:
	\begin{equation}\label{E:alpha-candidate}
	y(\tau) = \left\{ \begin{array}{ll}
	\frac{\Phi(\tau)}{\Phi(1)}, & 0 \leq \tau \leq 1\\
	1,  & \tau \geq 1.
	\end{array}\right.
	\end{equation}
The first thing we should check is that the function obtained this way, namely $m(f, \lambda) := y(\tau)$ satisfies the (discrete) main inequality \eqref{E:M-MI} of the function $\bM_b$. This is the content of the following lemma, which we prove shortly:
	
	\begin{lemma}\label{L:alphaMI}
	The function $m(f, \lambda) = y(\tau)$, where $\tau = f/\sqrt{\lambda}$ and $y$ is the function in \eqref{E:alpha-candidate},
	is a supersolution for \eqref{E:M-MI}.
	\end{lemma}
	
Obviously, this gives us that $M(f, \lambda) \leq m(f, \lambda)$. To see that we have,  in fact, equality, we consider a new variable:
	$$
	S := \Phi(\tau),
	$$
and observe that for a function $g$:
	\begin{equation}\label{E:DI-Svar}
	\big( \tau g'(\tau) + g''(\tau) \big) e^{\tau^2} = \frac{d^2g}{dS^2} = g_{SS}.
	\end{equation}
So \eqref{E:alpha-DI} is equivalent to $\alpha_{SS} \leq 0$, or $\alpha$ being concave in the variable $S$. It is easy to see that:

\textit{If $g(S)$ is a concave non-negative function for $S \geq 0$, then the ratio $\frac{g(S)}{S}$ is non-increasing.}
	
\noindent Thus, if we put $\alpha(\tau) := g(S)$, we have that for all $0 \leq \tau \leq 1$:
	$$\frac{g(S)}{S} = \frac{\alpha(\tau)}{\Phi(\tau)} \geq \frac{g(\Phi(1))}{\Phi(1)} = \frac{\alpha(1)}{\Phi(1)}
	= \frac{1}{\Phi(1)},$$
which gives exactly that $\bM_b(f, \lambda) \geq m(f, \lambda)$. Therefore
	\begin{equation}\label{E:bM-Bdry}
	\bM_b(|f|, \lambda) = \bM(f, |f|, \lambda) = \left\{ \begin{array}{ll}
		\frac{\Phi\left(\frac{|f|}{\sqrt{\lambda}}\right)}{\Phi(1)}, & |f| < \sqrt{\lambda} \\
		1, & |f| \geq \sqrt{\lambda}.
	\end{array}\right.
	\end{equation}

\begin{proof}[Proof of Lemma \ref{L:alphaMI}]
We define the quantities:
	\begin{equation}\label{E:X-def}
	\Xp := \frac{\tau + x}{\sqrt{1 - x^2}} \text{ and } \Xn := \frac{\tau - x}{\sqrt{1 - x^2}},
	\end{equation}
for all $\tau \geq 0$, and $0 \leq x <1$, $x\leq \tau$. 
We claim that, for all $0 \leq x \leq \tau < 1$, the function $\Phi$ satisfies
		\begin{equation}\label{E:Phi-MI}
		2\Phi(\tau) \geq \Phi(\Xp) + \Phi(\Xn).
		\end{equation}
In what follows, suppose $\tau \in [0, 1)$ is fixed, and we wish to show that
	$$ 2\Phi(\tau) \geq g(x), \:\forall\: 0 \leq x \leq \tau \text{, where }
		g(x) := \Phi(\Xp) + \Phi(\Xn). $$
Since $g(0) = 2\Phi(\tau)$, it suffices to show that $g$ is non-increasing. We have
	$$ \frac{d}{dx}\Xp = \frac{1+\tau x}{(1-x^2)^{3/2}} \text{ and } \frac{d}{dx}\Xn = -\frac{1-\tau x}{(1-x^2)^{3/2}},$$
and then
	\begin{align*}
	g'(x) \leq 0 &\Leftrightarrow \frac{1+\tau x}{1 - \tau x} \leq e^{\frac{2\tau x}{1-x^2}} \\
		& \Leftrightarrow 0 \leq G(x) \text{, where } G(x) = \frac{2\tau x}{1-x^2} - 
			\log\left(\frac{1+\tau x}{1 - \tau x}\right).
	\end{align*}
Since $G(0) = 0$, it suffices to show that $G$ is non-decreasing. A simple computation shows that
	$$G'(x) = 2\tau \left( \frac{1+x^2}{(1-x^2)^2} - \frac{1}{1 - x^2\tau^2} \right) \geq 0,
	\:\: \forall\: 0 \leq x \leq \tau < 1.$$
This completes the proof for \eqref{E:Phi-MI}. 

Returning to Lemma \ref{L:alphaMI}, recall that we wish to show that
	$$2m(f, \lambda) \geq m(f+a, \lambda - a^2) + m(f-a, \lambda -a^2), \:\forall\: f \geq a \geq 0, \:
		\lambda > a^2,$$
where $m(f, \lambda) = y(\tau)$, and $y(\tau) = \min(\Phi(\tau)/\Phi(1), 1)$, for $\tau = \frac{f}{\sqrt{\lambda}} \geq 0$.
Using the homogeneity of $m$, we can rewrite this in terms of $y$. Moreover, letting
$x := \frac{a}{\sqrt{\lambda}}$, we have that $0 \leq x <1$ and also $x \leq \tau$, so we may use 
exactly the quantities $\Xp$ and $\Xn$ defined in \eqref{E:X-def} to rewrite the inequality we have to prove:
	\begin{equation}\label{E:m-sup-1} 
	2y(\tau) \geq y(\Xp) + y(\Xn), \:\forall\: \tau\geq 0, 0 \leq x <1, x\leq \tau.
	\end{equation}
If $\tau < 1$, then it is easy to see that $\Xn \leq \tau < 1$, so \eqref{E:m-sup-1} becomes
	$$2\Phi(\tau) \geq \Phi(\Xn) + \Phi(1) y(\Xp).$$
If $\Xp < 1$, this becomes exactly \eqref{E:Phi-MI}. If $\Xp \geq 1$, the inequality follows again by \eqref{E:Phi-MI} and monotonicity of $\Phi$:
	$$\Phi(\Xn) + \Phi(1) \leq \Phi(\Xn) + \Phi(\Xp) \leq 2\Phi(\tau).$$
Finally, when $\tau \geq 1$, $y(\tau) = 1$, and since $y \leq 1$ always, 
$2 = 2y(\tau) \geq y(\Xp) + y(\Xn)$.

\end{proof}

%-------------------------------%
%------boundary case for bL-----%
\subsection{The boundary case $\bL(f, 1, \lambda)$}
Define
	$$ \bL_b(f, \lambda) := \bL(f, 1, \lambda) = \inf\{\La|\varphi|\Ra_J: \text{supp}(\varphi) \subset J; \: \La \varphi\Ra_J = f; \:
		S_J^2\varphi \geq \lambda \text{ a. e. on } J\}.
	$$
Some of the obvious properties $\bL_b$ inherits from $\bL$ are:
	\begin{itemize}
	\item Domain: $\Omega_{\bL_b} := \{(f, \lambda): f\in\bR; \lambda > 0\}$;
	\item $\bL_b$ is increasing in $\lambda$ and even in $f$;
	\item Homogeneity: $\bL_b(tf, t^2\lambda) = |t| \bL_b(f, \lambda)$;
	\item Range/Obstacle Condition: $|f| \leq \bL_b(f, \lambda) \leq \max\{|f|, \sqrt{\lambda}\}$;
	\item Main Inequality:
		\begin{equation} \label{E:L-MI}
		2\bL_b(f, \lambda) \leq \bL_b(f-a, \lambda - a^2) + \bL_b(f + a, \lambda - a^2), \:\: \forall \: |a| < \sqrt{\lambda}.
		\end{equation}
	\item $\bL_b$ is convex in $f$, and recall from \eqref{E:bL-min0} that $\bL_b$ is minimal at $f = 0$:
		\begin{equation} \label{E:L-min0}
		\bL_b(0, \lambda) \leq \bL_b(f, \lambda), \:\: \forall \: f,
		\end{equation}
	therefore $\bL_b$ is non-decreasing in $f$ for $f\geq 0$, and non-increasing in $f$ for $f\leq 0$;
	\item Greatest Subsolution: If $\ell(f, \lambda)$ is any continuous non-negative function on $\Omega_{\bL_b}$ which satisfies the 
	main inequality
		\begin{equation}\label{MIL}
		2\ell(f, \la) \le \ell(f+a, \la-a^2) +\ell(f-a, \la-a^2)
		\end{equation}
	and the range condition $\ell(f, \lambda) \leq \max\{|f|, \lambda\}$, then
		$\ell \leq \bL_b$. See Remark \ref{R:GrtSub-Bdry}.
	\end{itemize}

Using homogeneity, we write
	\begin{equation}
	\label{discrL}
	\bL_b(f, \lambda) = \sqrt{\lambda} \bL_b\left(\frac{f}{\sqrt{\lambda}}, 1\right) =: \sqrt{\lambda} b(\tau)
	\text{, where } \tau: = \frac{f}{\sqrt{\lambda}}. 
	\end{equation}
Then $b : \bR \rightarrow [0, \infty)$, $b$ is even in $\tau$, and from \eqref{E:L-min0}:
	\begin{equation}\label{E:beta-min0}
		b(0) \leq b(\tau), \:\: \forall \: \tau.
		\end{equation}
Moreover, $b$ satisfies
		\begin{equation}\label{E:beta-OC}
		b(\tau) = |\tau|, \:\: \forall \: |\tau| \geq 1.
		\end{equation}

In terms of $b$, \eqref{discrL} becomes

\begin{equation}
\label{discr_b}
\frac12\Big[ b\Big(\frac{\tau-h}{\sqrt{1-h^2}}\big)+b\Big(\frac{\tau-h}{\sqrt{1-h^2}}\big)\Big]-b(\tau) \ge 0\,.
\end{equation}

%%%%%%%%%%%%%%%%%%%%%

Since $\bL_b$ is concave in the first variable, we know that $b$ is concave.
	We will now use the second order   a.e. Taylor formula for concave functions
	from \cite{EG}:
\begin{equation}
\label{concA}
F(\tau+\eps) = F(\tau) + F'(\tau) \eps + \frac12 F''(\tau) \eps^2 + 
o(\eps^2), \quad \eps\to 0, \quad \text{for a.e.}\,\,x\,,
\end{equation}
We use this formula  for $F=b$ in conjunction with \eqref{discr_b}. We use also the  expansions:
$$
\frac{\tau \pm h}{\sqrt{1-h^2}} - \tau= \pm h +\frac12 \tau  h^2 + o(h^2)=:\eps\,.
$$

The inequality \eqref{discr_b} then obviously implies the following inequality:
	\begin{equation}\label{E:alpha-DI}
	b''(\tau) + \tau b'(\tau) -b(\tau) \geq 0.
	\end{equation}
	We just proved this inequality in a.e. sense. 
	
To pass to distributional sense, we notice that concave $b$ is everywhere defined and continuous. Its  derivative $b'$ is also its distributional derivative, and it is defined everywhere except for countably many jump points and it is a
decreasing function.  

Let $(b)''$ denote the distributional derivative of decreasing function $b'$. Thus it is a non-positive measure. We denote its singular part by symbol $\sigma_s$. 
Hence, in the sense of distributions
\begin{equation}
\label{concD}
(b)'' + \tau b'\, d\tau- b(\tau)\,d\tau= \big(b''(\tau) + \tau b'(\tau) - b(\tau)) \,d\tau + d \sigma_s \le 0\,.
\end{equation}
Hence, now we have  in the sense of distributions the following inequality (it will be used later in this sense):
\begin{equation}\label{E:beta-DI-distr}
	b''(\tau) + \tau b'(\tau) - b(\tau) \geq 0.
	\end{equation}

Since $b$ is even, we focus next only on $\tau \geq 0$. 

The general solution to the differential equation $z''(\tau) + \tau z'(\tau) - z(\tau) = 0$ for $\tau \geq 0$ is
	$$z(\tau) = C \Psi(\tau) + D\tau \text{, where } \Psi(\tau) = \tau\Phi(\tau) + e^{-\tau^2/2}, \:\: \forall \: \tau \geq 0.$$
Note that
	\begin{equation}
	\label{Phi}
	\Psi'(x) =\Phi(x),\,\, \Psi''(x) = e^{-x^2/2}\,.
	\end{equation}
Given our condition that $b(\tau) = \tau$ for all $\tau \geq 1$, a reasonable candidate for our function $b$ is one already proposed by Bollobas \cite{Bollobas}:
	\begin{equation}\label{E:beta-candidate}
	z(\tau) := \left\{ \begin{array}{ll}
		\frac{\Psi(\tau)}{\Psi(1)}, & 0 \leq \tau < 1\\
		\tau, & \tau \geq 1.
		\end{array}\right.
	\end{equation}
In other words, a candidate for $\bL_b$ is
	\begin{equation}
	\label{L1}
	L(f, \la)=\begin{cases}\sqrt{\lambda} \frac{\Psi\left(\frac{|f|}{\sqrt{\lambda}}\right)}{\Psi(1)},\, \sqrt{\la}\ge |f|,\\
	|f|,\, \sqrt{\la}\le |f|\,.\end{cases}
	\end{equation}

Our first goal will be to prove:

\begin{lemma}
\label{subs-2}
The function $L$ defined in \eqref{L1} satisfies \eqref{MIL}.
\end{lemma}

Since it is easy to verify that $L$ satisfies the range condition $L(f, \lambda) \leq \max\{|f|, \sqrt{\lambda}\}$, we have then that $L$ is a subsolution of \eqref{MIL}, and so
	$$L \leq \bL_b.$$
	
Now we want to prove the opposite inequality
\begin{equation}
\label{bLsm}
\bL_b\le L.
\end{equation}
Recall that we write $\bL_b(f, \lambda) = \sqrt{\lambda} b(\tau)$, where $\tau = \frac{f}{\sqrt{\lambda}}$. We look only at $\tau \geq 0$.
Consider again a new variable
	\begin{equation}\label{E:Tvar-def}
	T := \frac{\tau}{\Psi(\tau)}, \:\: \tau \geq 0.
	\end{equation}
Then
	$$ \frac{dT}{d\tau} = \frac{e^{-\tau^2/2}}{\Psi^2(\tau)}, $$
which shows that $T$ is strictly increasing in $\tau$. Moreover, it is easy to check that for a function $g$, we have
	\begin{equation}\label{E:DI-Tvar}
	\frac{d^2}{dT^2} \left( \frac{g(\tau)}{\Psi(\tau)} \right) = \Psi^3(\tau) e^{\tau^2} (g'' + \tau g' - g).
	\end{equation}
So, if we circle back to our function $b$, and denote
	$$\beta(T) := \frac{b(\tau)}{\Psi(\tau)},$$
the infinitesimal main inequality \eqref{E:beta-DI} for $b$ is equivalent to $\beta_{TT} \geq 0$, or $\beta$ being convex in the variable $T$. Now note that
	$$ \beta'(T) = \bigg(b'(\tau) \Psi(\tau) - b(\tau) \Phi(\tau)\bigg) e^{\tau^2/2}. $$
Since $T = 0$ only at $\tau = 0$, we have 
	$$\beta'(T)|_{T \rightarrow 0_{+}} = b'(0_{+}) \geq 0,$$
where $b'(0_{+})$ denotes the right derivative of $b$ at $0$. This is non-negative because $b$ is a convex, even function.
%where the last equality follows from \eqref{E:beta-min0} ($b$ has a global minimum at $\tau = 0$). 
So now we have that $\beta(T)$ is convex and $\beta'(0_{+}) \geq 0$, showing that $\beta$ is \textit{non-decreasing} for $T\geq 0$. 
Finally, we have then that for any $0 \leq \tau < 1$:
	$$ \frac{b(\tau)}{\Psi(\tau)} \leq  \frac{b(1)}{\Psi(1)} = \frac{1}{\Psi(1)},$$
therefore
	$$ b(\tau) \leq \frac{\Psi(\tau)}{\Psi(1)}, \:\: \forall \: \tau \in [0, 1], $$
which is exactly $\bL_b \leq L$. So Theorem \ref{T:Lb-Bdry} is proved, provided we have Lemma \ref{subs-2}, which we prove next.

\begin{proof}[Proof of Lemma \ref{subs-2}]
In fact, the proof is given in \cite{Bollobas}. It is slightly sketchy and leaves some cases to the reader, so here we follow the proof of \cite{Bollobas} in more details. The proof is divided into several cases.   By symmetry we can always  think that $f\ge  0$ in all cases.  Using the homogeneity we can always assume that $\lambda=1$. 

Case 1) will be when  both points $(f\pm a, 1-a^2)$ lie in $\Omega_{par} :=\{ (p,q) \in \mathbb{R}^{2}\, : q \geq p^{2}\}$.  Clearly then $(f,1)$ will be also in $\Omega_{par}$. 

Notice that $L(f, 1) =\max (\frac{\Psi\left(|f|\right)}{\Psi(1)}, |f|) =\frac{\Psi\left(|f|\right)}{\Psi(1)}$ if  $(f, 1) \in \Omega_{par}$. 

 Put 
\begin{equation}
\label{X}
X(f, a):= \frac{|f+a|}{(1-a^2)^{1/2}},\,\, a\in [-1, 1],\,\, f\in [0, 1)\,.
\end{equation}
Then \eqref{MIL} in our case can be rewritten as 
\begin{equation}
\label{bo1}
2\Psi(f) \le [\Psi(X(f, a)) +\Psi(X(f, -a))]\sqrt{1-a^{2}}.
\end{equation}
 Next,  without loss of generality assume that $a\geq  0$. The inequality is true for $a=0$.

 Let us check that 
\begin{equation}
\label{bo2}
\frac{\pd}{\pd a} \left( \sqrt{1-a^{2}}(\Psi(X(f, a)) +\Psi(X(f, -a)))\right) \ge 0\,.
\end{equation}
Consider the case when $f-a \geq 0$. Notice that 
\begin{align*}
&\frac{\partial}{\partial a} X(f,a) = \frac{1}{\sqrt{1-a^{2}}} + X(f,a) \frac{a}{1-a^{2}};\\
&\frac{\partial}{\partial a} X(f,-a) = -\frac{1}{\sqrt{1-a^{2}}} + X(f,-a) \frac{a}{1-a^{2}}.
\end{align*}

Using the fact that $\Psi'(s) = \Phi(s)$, $\Psi(s) =s\Psi'(s)+e^{-s^{2}/2}$, we get the equality
\begin{align*}
&\frac{\pd}{\pd a} (\Psi(X(f, a)) +\Psi(X(f, -a))) = \frac{1}{\sqrt{1-a^2}} ( \Phi(X(f, a)) - \Phi(X(f, -a)))+ \\
&\frac{a}{1-a^{2}}\left[ \Psi(X(f,a)) - \exp(-(X(f,a))^{2}/2) + \Psi(X(f,-a))- \exp(-(X(f,-a))^{2}/2) \right].
\end{align*}
Therefore
\begin{align*}
&\frac{\pd}{\pd a} \left( (\Psi(X(f, a)) +\Psi(X(f, -a)))\sqrt{1-a^{2}} \right) =( \Phi(X(f, a)) - \Phi(X(f, -a))) \\
&-\frac{a}{(1-a^2)^{1/2}}( e^{-X(f, a)^2/2} +  e^{-X(f, -a)^2/2} )\,.
\end{align*}
But $\frac{a}{(1-a^2)^{1/2}} = \frac12 (X(f, a)- X(f, -a))$, so  to prove \eqref{bo2} one needs to check the following inequality:
\begin{equation}
\label{bo3}
\frac1{X(f, a)-X(f, -a)}\int_{X(f, -a)}^{X(f, a)} e^{-s^2/2} ds \ge \frac12( e^{-X(f, a)^2/2} + e^{-X(f, -a)^2/2})\,.
\end{equation}
This inequality holds because in our case 1) we have $X(f, -a)\in [0,1], X(f, a)\in [0,1]$, and the function $s\mapsto e^{-s^2/2}$ is concave on the interval $[-1,1]$. It is easy to verify that for every concave function on an interval, its integral average over the interval is at least its average over the endpoints of the interval.

If $f-a \leq 0$, then $\frac{\partial}{\partial a} X(f,-a) = \frac{1}{\sqrt{1-a^{2}}} + X(f,-a) \frac{a}{1-a^{2}}$. Repeating the previous calculations verbatim eventually one will need to show the following inequality 
\begin{align*}
\Phi(X(f,a))+\Phi(X(f,-a)) \geq \frac{X(f,a)+X(f,-a)}{2}\left( e^{-X(f, a)^2/2} + e^{-X(f, -a)^2/2}\right),
\end{align*}
which is also true. Indeed, we want to show that $\Phi(a)+\Phi(b) \geq \frac{a+b}{2}(e^{-a^{2}/2}+e^{-b^{2}/2})$ for all $a,b \in [0,1]$. If $a=b$, then the inequality follows because $w(a):= \Phi(a)-ae^{-a^{2}/2}$, $w(a)\geq 0$  at $a=0$ is true, and its derivative is $a^{2}e^{-a^{2}/2}\geq 0$. In general, consider the map 
\begin{align*}
a \mapsto \Phi(a) + \Phi(b) - \frac{a+b}{2} (e^{-a^{2}/2}+e^{-b^{2}/2}) \quad \text{for} \quad a \in [b,1].
\end{align*}
The derivative of this map is $ \frac{1}{2} (e^{-a^{2}/2}  - e^{-b^{2}/2}) + \frac{a+b}{2}\cdot ae^{-a^{2}/2}$ which at point $a=b$ has a nonnegative sign.  Differentiating again we obtain $\frac{e^{-a^{2}/2}}{2}(1-a^{2})(a+b)\geq 0$. This finishes the proof of the  case 1).

\vskip0.5cm

Next,  consider Case 2): when $(f,1) \notin \Omega_{par}$. Then notice that $f  \mapsto L(f,1)$ is convex  as a maximum of two convex functions. Therefore 
\begin{align*}
\frac{1}{2}\left( L(f+a, 1-a^{2}) + L(f-a, 1-a^{2})\right) \geq L(f, 1 -a^{2})  = L(f,1).
\end{align*}

Case 3). Now suppose that  $(f \pm a, 1-a^2)$ are not in  $\Omega_{par}$  and $(f, 1)$ is in $\Omega_{par}$.   We remind that we are considering only $f\ge 0$.
Since $a  \mapsto |f+a|+|f-a|$ is increasing as $a$ increases, it suffices to consider the case when $(f-a,1-a^{2})$  is such that  $(f-a)^{2}=1-a^{2}$, i.e., the left point is on the parabola. Then we need to show that 
\begin{align}\label{patulya1}
2 \frac{ \Psi(f)}{\Psi(1)}\leq |f-a|+f+a.
\end{align}
%Denote $T=\frac{t}{\sqrt{\lambda}}$ and $X=X(T) = \frac{x}{\sqrt{\lambda}}$. 
Clearly $0\leq a\leq 1$. Consider the case when  $0\leq f\leq a$. From $(f-a)^{2}=1-a^{2}$ we obtain that $a-\sqrt{1-a^{2}}=:f(a) \geq 0,$ so $a\geq \frac{1}{\sqrt{2}}$, and the inequality (\ref{patulya1}) simplifies to 
$$
f(a)\leq \Psi^{-1}(\Psi(1)a), \quad 1\geq a\geq  \frac{1}{\sqrt{2}}.
$$
The left hand side is convex and the right hand side is concave (as an inverse of increasing convex function). Since at $t=1$ and $t=\frac{1}{\sqrt{2}}$ the inequality holds then it holds on the whole interval $[1/\sqrt{2},1]$. 

If $f\geq a$, then the condition $(f-a)^{2}=1-a^{2}$ implies that $f = a+\sqrt{1-a^{2}} \geq 1$ for all $a \in [0,1]$. Therefore the inequality (\ref{patulya1}) becomes $\frac{\Psi(f)}{\Psi(1)}\leq f$ which is correct if $f \geq 1$.  Indeed, consider $g(s) = \frac{\Psi(s)}{\Psi(1)} -s$. Then $g(1)=0$, $g'(1)<0$, and $g''(t)\geq 0$. Also $\lim_{s \to \infty} \frac{g(s)}{s} = \frac{\int_{0}^{\infty} e^{-t^{2}/2}dt}{\Phi(1)+\exp(-1/2)} -1 =-0.1428...<0$. This implies that $g(t)\leq 0$ for all $t\geq 1$.

\vskip0.5cm

Case 4a). Next we consider the case when $(f,1)$ is in $\Omega_{par}$, $(f+a,1-a^{2})$ is not in $\Omega_{par}$,  $(f-a,1-a^{2})$  is in $\Omega_{par}$ and it has non-negative first coordinate, i.e., $f-a\geq 0$ (the remaining case with negative first coordinate will be treated in Case 4b)). 

First consider the case when $f-a=0$, i.e., the first coordinate of the left point is zero. Then $f=a$. Since the right point is outside (below) of the parabola we have $\frac{f+a}{\sqrt{1-a^{2}}} = \frac{2a}{\sqrt{1-a^{2}}}\geq 1$. The latter means that $a \in [\frac{1}{\sqrt{3}},1]$. Then we need to show that 
\begin{align*}
2\Psi(a) = 2 \Psi(f) \leq \sqrt{1-a^{2}} \Psi\left(\frac{f-a}{\sqrt{1-a^{2}}}\right) + \Psi(1)(f+a) = \sqrt{1-a^{2}}+2\Psi(1)a.
\end{align*}
The left hand side of the inequality is convex. The right hand side of the inequality is concave. Inequality clearly holds for the endpoint cases, i.e., $a=1$ and $a=\frac{1}{\sqrt{3}}$. Therefore it holds in general. 

Notice that if $(f-a)^{2}=\lambda-a^{2}$ then we are in Case 3). So if we show that the map $a \mapsto L(f+a,1-a^{2})+L(f-a,1-a^{2})$ is concave when $1\geq \frac{f-a}{\sqrt{1-a^{2}}}\geq 0$ (left point is $\Omega_{par}$ with non-negative first coordinate), $f\leq 1$ (the point $(f,1)$ is in $\Omega_{par}$), and $\frac{f+a}{\sqrt{1-a^{2}}}\geq 1$ (the right point is not in $\Omega_{par}$) then this will prove  Case 4a) completely, because the concave function dominates the number $2L(f,1)$ at the endpoints of an interval. We have 
\begin{align*}
L(f+a,1-a^{2})+L(f-a,1-a^{2}) = \sqrt{1-a^{2}} \Psi\left(\frac{f-a}{\sqrt{1-a^{2}}}\right) + \Psi(1)(f+a).
\end{align*}
The second term is linear in $a$. Its first derivative is 
\begin{align*}
-\Phi\left(\frac{f-a}{\sqrt{1-a^{2}}} \right)  - \frac{a}{\sqrt{1-a^{2}}} \exp\left(-\left[\frac{f-a}{\sqrt{1-a^{2}}}\right]^{2}/2\right)+\Psi(1).
\end{align*}

Its second derivative is 
\begin{align*}
\frac{a(a+af^{2}-2f)}{(1-a^{2})^{5/2}}\exp\left(-\left[\frac{f-a}{\sqrt{1-a^{2}}}\right]^{2}/2\right).
\end{align*}
The map $a \mapsto a+af^{2}-2f$ is increasing in $a$. Let us increase $a$. Two scenarios can occur: 1) $f-a=0$ or 2) $\frac{f-a}{\sqrt{1-a^{2}}}=1$.  In the first case we get $a+af^{2}-2f = f(f^{2}-1)\leq 1$ since $0\leq f\leq 1$. In the second case the condition $a \in [0,1]$ implies
$$
a+af^{2}-2f = -a-2\sqrt{1-a^{2}}+a^{2}(\sqrt{1-a^{2}}+a)^{2} = a(a-1)+2\sqrt{1-a^{2}}(a^{3}-1)\leq 0.
$$
 Thus in all cases we obtain $a+af^{2}-2f \leq 0$, therefore this finishes the proof of the case 4a). 

\vskip0.5cm

Case 4b). 
 It remains to show that  if the right point already left $\Omega_{par}$ but the left point is in $\Omega_{par}$ with negative first coordinate,  then \eqref{MIL} still holds. Then the required inequality amounts to 
\begin{align*}
2\Psi(f) \leq \sqrt{1-a^{2}} \Psi\left(\frac{a-f}{\sqrt{1-a^{2}}}\right)+\Psi(1)(f+a),
\end{align*}
where  $|\sqrt{1-a^{2}}-a|\leq f\leq a\leq 1$ (notice that the latter inequality simply means that $\frac{f+a}{\sqrt{1-a^{2}}}\geq 1$, i.e., the right point is not in $\Omega_{par}$, and $\frac{a-f}{\sqrt{1-a^{2}}}\leq 1$, the left point is in $\Omega_{par}$ with negative first coordinate). It is the same as to show 
\begin{align}\label{erti1}
 \Psi\left(\frac{a-f}{\sqrt{1-a^{2}}}\right) + \Psi(1)\left(\frac{a-(\frac{2\Psi(f)}{\Psi(1)}-f)}{\sqrt{1-a^{2}}}\right)\geq 0
\end{align}
for all $0\leq f\leq 1$ if $ \max\{f, \frac{\sqrt{2-f^{2}}-f}{2}\}\leq  a\leq  \frac{f+\sqrt{2-f^{2}}}{2}$. 

 Let as show that the derivative in $a$ of the left hand side of (\ref{erti1}) is nonnegative. If this is the case then we are done because by increasing $a$ we can reduce the inequality to an endpoint case which is already verified. $\Psi$ is increasing (see \eqref{Phi}), and since $fa\leq 1$ therefore $a \mapsto \Psi\left(\frac{a-f}{\sqrt{1-a^{2}}}\right)$, $a \in  [f,1]$ is increasing as a composition of two increasing functions.
 Here we have used the fact that 
 \begin{align*}
 \frac{\partial}{\partial a}\left( \frac{a-f}{\sqrt{1-a^{2}}}\right) = \frac{1-af}{(1-a^{2})^{3/2}}.
 \end{align*}
  To check the monotonicity  of the map $a\mapsto \frac{a-(\frac{2\Psi(f)}{\Psi(1)}-f)}{\sqrt{1-a^{2}}}$ it is enough to verify that $a(\frac{2\Psi(f)}{\Psi(1)}-f) \leq 1$. The latter inequality follows from the following two simple inequalities 
\begin{align}
&\Psi(f) \geq  \frac{\Psi(1) f}{2}, \quad 0\leq f\leq 1, \label{ori}\\
&\left(\frac{f+\sqrt{2-f^{2}}}{2} \right) \left(\frac{2\Psi(f)}{\Psi(1)}-f\right)\leq 1, \quad  0\leq x\leq 1. \label{sami}
\end{align}
Indeed, to verify (\ref{ori}) notice that  $\frac{d}{df}\frac{\Psi(f)}{f} = \frac{f\Phi(f)-\Psi(f)}{f^{2}} = -\frac{e^{-\frac{f^{2}}{2}}}{f^{2}} <0$, therefore $\frac{\Psi(f)}{f}\geq \Psi(1) \geq \frac{\Psi(1)}{2}$. 

To verify (\ref{sami}) it is enough to show that 
$$
\frac{\Psi(x)}{\Psi(1)x}\leq \frac{1}{x^{2}+x\sqrt{2-x^{2}}}+\frac{1}{2}, \quad x \in [0,1].
$$
If $x=1$ we have equality. Taking derivative of the mapping $x \to  \frac{\Psi(x)}{\Psi(1)x} -\frac{1}{x^{2}+x\sqrt{2-x^{2}}}-\frac{1}{2}$ in $x$ we obtain 
$$
\frac{2}{x^{2}}\left(-\frac{e^{-\frac{x^{2}}{2}}}{2\Psi(1)}+\frac{x+\frac{1-x^{2}}{\sqrt{2-x^{2}}}}{(x+\sqrt{2-x^{2}})^{2}}\right)\geq 0.
$$
To prove the last inequality it is the same as to show that $\frac{\sqrt{2-x^{2}}+x(2-x^{2})}{x\sqrt{2-x^{2}}+1-x^{2}}\leq \Psi(1) e^{\frac{x^{2}}{2}}$. For the exponential function we use the estimate $e^{\frac{x^{2}}{2}}\geq 1+\frac{x^{2}}{2}$. We estimate $\sqrt{2-x^{2}}$ from above in the numerator by $\sqrt{2}(1-\frac{x^{2}}{4})$, and we estimate $\sqrt{2-x^{2}}$ from below in the denominator by $(1-\sqrt{2})(x-1)+1$ (as $x\to\sqrt{2-x^{2}}$ is concave). Thus it would be enough to prove that 
$$
\frac{\sqrt{2}(1-\frac{x^{2}}{4})+x(2-x^{2})}{\sqrt{2}x(1-x)+1}\leq \Psi(1) \left(1+\frac{x^{2}}{2} \right), \quad 0\leq x\leq 1.
$$  
If we further use the estimates  $\Psi(1) \geq \frac{29}{28}$,  and $\frac{41}{29}\leq \sqrt{2}\leq \frac{17}{12}$ (for  denominator and numerator correspondingly), then the last  inequality would follow  from  
$$
\frac{29}{240} \cdot \frac{246x^{4}-486x^{3}+233x^{2}-12x-8}{29+41x-41x^{2}}\leq 0.
$$
The denominator has the positive sign. The negativity of $246x^{4}-486x^{3}+233x^{2}-12x-8\leq 0$ for $0\leq x\leq 1$ follows from the Sturm's algorithm,  which shows that the polynomial does not have roots on $[0,1]$. Since at point $x=0$ it is negative therefore it is negative on the whole interval.  
\bigskip 
\end{proof}

\end{document}